\documentclass{amsart}
\usepackage{amssymb}
\usepackage{verbatim}
\usepackage{array}
\usepackage{latexsym}
\usepackage{enumerate}
\usepackage{tikz}
\usepackage{float}
\usepackage{amsmath}
\usepackage{amsfonts}
\usepackage{amsthm}
\usepackage{color}
\usepackage{longtable}
\usepackage[english]{babel}

\newtheorem{theorem}{Theorem}[section]
\newtheorem{proposition}[theorem]{Proposition}
\newtheorem{lemma}[theorem]{Lemma}

\newtheorem{remark}[theorem]{Remark}

\def\cP{\mathcal P}

\def\fqs{\mathbb{F}_{q^2}}
\def\fqt{\mathbb{F}_{q^3}}

\def\PG{{\rm{PG}}}
\def\GL{{\rm{GL}}}
\def\SL{{\rm{SL}}}
\def\ord{\mbox{\rm ord}}

\def\Alt{\mbox{\rm Alt}}
\def\Sym{\mbox{\rm Sym}}

\def\fq{{\mathbb F}_q}

\newcommand{\PSL}{\mbox{\rm PSL}}
\newcommand{\PGL}{\mbox{\rm PGL}}

\newcommand{\PSU}{\mbox{\rm PSU}}

\newcommand{\aut}{\mbox{\rm Aut}}

\newcommand{\diag}{\mbox{\rm diag}}

\title{The M\"obius function of $\PSL(3,2^{p})$ for any prime $p$}
\date{}
\author{Martino Borello, Francesca Dalla Volta, Giovanni Zini}

\begin{document}



\begin{abstract}
Let $G$ be the simple group ${\rm PSL}(3,2^{p})$, where $p$ is a
prime number. For any subgroup $H$ of $G$, we compute the M\"obius
function of $H$ in the subgroup lattice of $G$. To
this aim, we describe the intersections of maximal subgroups of $G$.
We point out some connections of the M\"obius function with
other combinatorial objects, and, in this context, we compute the
reduced Euler characteristic of the order complex of the subposet of
$r$-subgroups of ${\rm PGL}(3,q)$, for any prime $r$ and any prime
power $q$.
\end{abstract}

\maketitle

\begin{small}



\end{small}

\section{Introduction}

Let $G$ be a finite group. The M\"obius function of $G$ is defined
recursively by $\mu(G)=1$ and $\mu(H)=-\sum_{K\,:\,H<K\leq G}\mu(K)$ for any $H<G$. 
It was introduced independently by Weisner \cite{Weisner} and Hall \cite{Hall}; in
particular, Hall provides a formula to enumerate generating tuples of elements of $G$.
For any group $K$, let $\sigma_n(K)$ and
$\phi_n(K)$ denote respectively the number of ordered $n$-tuples and generating $n$-tuples of
elements of $K$. 
Then $ \sigma_n(G)=\sum_{H:H\leq G}\phi_n(H)$, and by the M\"obius
inversion formula (\cite[Prop. 3.7.1]{Stanley}) one gets $
\phi_n(G)=\sum_{H:H\leq G}\sigma_n(H)\cdot\mu(H)$.
%
%


Let $\it {Prob}_G(n)$ be the probability that $n$ random elements of $G$ generate $G$. Then
$$\it {Prob}_G(n)=\frac {\phi_n(G)}{\mid G\mid ^n}=\sum_{H\leq G}\frac{\mu _G(H)}{[G:H]^n}.$$
It is possible to define the following complex function (see \cite{Mann1}):
$${P}_G(s)=\sum_ {n\in \mathbb{N}}\frac {a_n(G)}{n^s},\, s\in \mathbb{C},\, \, \textnormal{
where}\,\,  a_n(G)=\sum _{[G:H]=n }\mu _G(H).$$
It satisfies ${P}_G(t)= \it {Prob}_G(t)$ for $t\in \mathbb{N}$.

In case of a profinite group $G$, the above function is generalized as follows:
$$P_G(s)=\sum _{H\leq_{o}G}\frac {\mu _G (H)}{[G:H]^s},$$
where $H$ ranges over all open subgroups of $G$. Mann conjectured
that this sum is absolutely convergent in some half complex plane
for {\it {positively finitely generated (PFG) groups}}. This
conjecture is verified if the following two facts hold for any PFG
profinite group (see \cite{Mann1,Mann2}): $|\mu(H)|$ is bounded by a
polynomial function in $[G:H]$; the number of subgroups $H$ of index
$n$ satisfying $|\mu(H)|\ne0$ grows at most polynomially in $n$.
The conjecture was reduced by Lucchini \cite{Lucchini} to the following one:
there exist $c_1,c_2\in\mathbb{N}$ such that, for any almost simple group $G$, $|\mu(H)|\leq [G:H]^{c_1}$ for any $H<G$; and, for any $n\in\mathbb{N}$, the number of subgroups $H<G$ of index $n$ in $G$ satisfying $G=H\,{\rm soc}(G)$ and $\mu(H)\ne0$ is upper bounded by $n^{c_2}$.
This was proved in \cite{LucchiniColombo} in the case of alternating and symmetric groups.

Not very much is known about  the exact values of $\mu(H)$ when $G$ is a
simple group; up to our knowledge, the only infinite families of
non-abelian simple groups for which the M\"obius function is
completely known are the following.
\begin{itemize}
\item The groups $\PSL(2,q)$; for $q$ prime see \cite{Hall}, for any prime power $q$ see \cite{Downs}, where also the groups $\PGL(2,q)$ are completely worked out.
\item The Suzuki groups $Sz(q)$ for any non-square power $q$ of $2$; see \cite{DownsJones2016}.
\item The Ree groups $Ree(q)$ for any non-square power $q$ of $3$; see \cite{Pierro}.
\item The $3$-dimensional unitary groups $\PSU(3,2^{2^n})$ for any $n>0$; see \cite{Zini}.
\end{itemize}
For any of these families Mann's conjecture is verified. In this
paper we consider the case of $3$-dimensional projective general
linear groups $\PGL(3,q)$; again, Table \ref{tabella:mobius}
confirms Mann's conjecture when $q=2^p$ with prime $p$.

The main result of this paper, Theorem \ref{th:mobius}, provides the M\"obius function of the simple group $G=\PSL(3,2^p)$ for any odd prime $p$ (note that $G=\PGL(3,2^p)$). The subgroups with non-zero
M\"obius functions are summarized in Tables \ref{tabella:mobius}. In
the second column we specify the Aschbacher's classes in which $H$
is contained, with an N when $H$ is not maximal in $G$ \cite{low};
for example $\mathcal{C}_1,\mathcal{C}_2$ (N) means that $H$ is
intersection of maximal groups in $\mathcal{C}_1$ and
$\mathcal{C}_2$ and $H$ is not maximal.
For the sake of
completeness, Table \ref{tabella:mobius4} summarizes the case $p=2$, which can be easily computed using GAP, via the full table of marks of $\PSL(3,4)$.

We observe that, when considering the group $\PSL(3,2^n)$ with a non-prime $n$, other maximal subgroups appear, such as subgroups isomorphic to $\PSL(3,2^{n^\prime})$, or $\PGL(3,2^{n^\prime})$, or $\PSU(3,2^{n^\prime})$, or $\PG(3,2^{n^\prime})$, for certain divisors $n^\prime$ of $n$.
This makes the computations quite longer and not easily tractable. 

Throughout the paper we will use the following group-theoretic
notation which is based on the ATLAS \cite{ATLAS}: given two
subgroups $H$ and $K$ of $G$, $H\times K$ denotes their direct
product, $H:K$ a split extension of $H$ by $K$, and $H.\,K$ a
(split or non-split) extension of $H$ by $K$; $C_n$ a cyclic group
of order $n$; $D_n$ a dihedral group of order $n$;  $E_{p^n}$ an
elementary abelian group of order $p^n$; $H^{m+n}$ denotes the
extension $H^m.\, H^n$; $\Sym(n)$ and $\Alt(n)$ denote
respectively the symmetric and the alternating group of degree $n$;
${\rm Syl}_r(G)$ denotes a Sylow $r$-subgroup of the group $G$ under
consideration. For simplicity, we will use $S_r$ for ${\rm
Syl}_r(G)$; when $q$ is a power of $2$, $S_2\cong E_q^{1+2}$.

In order to prove Theorem \ref{th:mobius}, the intersections of
maximal subgroups of $G$, their conjugacy classes and normalizers
are carefully investigated and determined; for each subgroup, the
M\"obius function is computed. To this aim, we apply geometric
arguments regarding the geometry of $G$ and its subgroups in their
natural action on the plane $\PG(2,q)$ over $\mathbb{F}_q$, and more
generally in their action on the plane
$\PG(2,\overline{\mathbb{F}}_q)$ over the algebraic closure of
$\mathbb{F}_q$.


Finally, we point out that the M\"obius function of a finite group
$G$ has connections with different areas of
mathematics, in which the M\"obius inversion formula turns out to be
applicable. We list some objects whose enumeration can be performed by means of the M\"obius function of $G$.
\begin{enumerate}
\item Epimorphisms from a free group of finite rank to the group $G$; see \cite{DownsJones2016}.
\item Graphs $\tilde{\Gamma}$ which are a $G$-covering of a given graph $\Gamma$; see \cite{LLK}.
\item The structure of the group of units of the monoid of cellular automata over the configuration space $A^G$, for a given finite set $A$; see \cite[Section 4]{Castillo}.
\item Reduced Euler characteristic of the order complex of posets $\cP$ associated to $G$.
\end{enumerate}

We will explicitly work out the computation of point (4) in Section \ref{sec:connections}, where we consider the order complex of the finite poset $\cP=L_r$ of $r$-subgroups of $\PGL(3,q)$ ordered by inclusion, for any prime power $q$ and any prime $r$.
The results are summarized in Table \ref{tabella:euler}.

This paper is organized as follows. Section \ref{sec:prelim} contains preliminary results on the M\"obius function of a finite group, and on the groups $\PGL(3,q)$. In Section \ref{sec:mainresult} the main result on the M\"obius function of $\PSL(3,2^p)$ is stated, namely Theorem \ref{th:mobius}. Section \ref{sec:proof} provides the proof of Theorem \ref{th:mobius}. Finally, Section \ref{sec:connections} computes the reduced Euler characteristic of certain ordered complexes associated to $\PGL(3,q)$.

\section{Preliminary results}\label{sec:prelim}

For any locally finite poset $(\cP,\preceq)$, define the M\"obius function
$\mu_{\cP}:\cP\times\cP\to\mathbb{Z}$ by
$$ \mu_\cP(x,y)= 0 \ \textrm{if} \ x\not\preceq y,
\quad \mu_\cP(x,x)=1,\quad \mu_\cP(x,y)=-\sum_{z\in\cP\,:\,x\prec
z\preceq y}\mu_\cP (z,y) \ \textrm{if}\ x\prec y. $$
For $x\prec y$, $\mu_{\cP}$ is equivalently defined by $
\mu_{\cP}(x,y)=-\sum_{z\in\cP\,:\,x\preceq z\prec y}\mu_{\cP}(x,z)$.
We will consider the poset $\cP=L$ of subgroups of a finite group
$G$, ordered by inclusion; $L$ is a lattice with greatest element
$G$ and least element $\{1\}$. For simplicity, we denote by $\mu(H)$
the M\"obius function of $H\leq G$. The function
$\mu:L\to\mathbb{Z}$, $H\mapsto\mu(H)$ will be called the
\emph{M\"obius function} of $G$.

Clearly, if $H,K\leq G$ are conjugated, then
$\mu(H)=\mu(K)$. The following property restricts
the investigation to the intersections of maximal subgroups of $G$.
\begin{theorem}\label{th:nonzero}\textrm{(\cite[Theorem 2.3]{Hall})}
If $H\leq G$ satisfies $\mu(H)\ne0$, then $H$ is the intersection of
maximal subgroups of $G$.
\end{theorem}




Let $q$ be a prime power; we consider the group $\PGL(3,q)$.
Note that, when $q=2^p$ with $p$ an odd prime,
$\PGL(3,q)=\PSL(3,q)$; $q=2^p$ implies also
$\PSL(3,q)=\SL(3,q)$, which allows us to use matrices to denote the
elements of $\PSL(3,q)$.

The classification of subgroups of $\PGL(3,q)$ goes back to Mitchell
\cite{Mitchell} and Hartley \cite{Hartley} . We refer to \cite{Mitchell,Hartley,King} for the proof of the
following classical results, and to \cite{HP} for a general
reference on projective planes.


\begin{theorem}\label{th:psl3q}
For any prime power $q$, the following are self-normalizing maximal subgroups of
$\PGL(3,q)$, and they are unique up to conjugation:
\begin{enumerate}
\item the stabilizer $E_{q^2}:\GL(2,q)$ of an $\fq$-rational point, of order $q^3(q-1)^2(q+1)$;
\item the stabilizer $E_{q^2}:\GL(2,q)$ of an $\fq$-rational line, of order $q^3(q-1)^2(q+1)$;
\item the stabilizer $(C_{q-1})^2: {\rm Sym}(3)$ of an $\fq$-rational triangle, of order $6(q-1)^2$;
\item the stabilizer $C_{q^2+q+1}: C_3$ of an $\fqt\setminus\fq$-rational triangle, of order $3(q^2+q+1)$.
\end{enumerate}
If $q=2^p$ with $p$ an odd prime, the only other maximal subgroup of
$\PGL(3,q)$ up to conjugation is the following:
\begin{itemize}
\item the stabilizer of a subplane of order $2$, of order $168$ and isomorphic to $\PSL(3,2)$.
\end{itemize}
\end{theorem}

It follows immediately that every Sylow subgroup of
$\PGL(3,q)$ in contained in one of the maximal subgroups (1) to (4)
in Theorem \ref{th:psl3q}.
For the reader's convenience, we recall in Remark
\ref{remark:action} which points, lines or triangles in
$\PG(2,\overline{\mathbb{F}}_q)$ are stabilized by any element
$\sigma\in\PGL(3,q)$, in terms of $\ord(\sigma)$.

\begin{remark}\label{remark:action}
Let $q$ be a power of a prime $r$ and
$\sigma\in\PGL(3,q)\setminus\{1\}$. Then one of the following cases
holds.
\begin{itemize}
\item $\ord(\sigma)=r$ and $\sigma$ is an elation, i.e. $\sigma$ stabilizes every line through an $\fq$-rational point $C$ and every point of an $\fq$-rational line $\ell$ passing through $C$; $C$ and $\ell$ are called the center and the axis of $\sigma$.
\item $\ord(\sigma)=r\ne2$, or $r=2$ and $ord(\sigma)=4$. Also, $\sigma$ stabilizes exactly one point $P$ and one line $\ell$; both $P$ and $\ell$ are $\fq$-rational, and $P\in\ell$.
\item $\ord(\sigma)\mid(q-1)$ and $\sigma$ is a homology, i.e. $\sigma$ stabilizes every line through an $\fq$-rational point $C$ and every point of an $\fq$-rational line not passing through $C$; $C$ and $\ell$ are the center and the axis of $\sigma$.
\item $\ord(\sigma)=r\cdot d$ with $1\ne d\mid(q-1)$; $\sigma$ stabilizes two $\fq$-rational points $C$ and $P$, the line $CP$, and another $\fq$-rational line passing through $P$.
\item $2\ne\ord(\sigma)\mid(q-1)$ and $\sigma$ stabilizes three non-collinear $\fq$-rational points $P,Q,R$ and the lines $PQ,PR,QR$.
\item $\ord(\sigma)\mid(q^2-1)$ and $\ord(\sigma)\nmid(q-1)$. Also, $\sigma$ stabilizes an $\fq$-rational point $P$ and two $\fqs\setminus\fq$-rational points $Q,R$ which are conjugated under the $\fq$-Frobenius collineation: $Q^q=R$, $R^q=Q$; $\sigma$ stabilizes the $\fq$-rational line $QR$ and the $\fqs\setminus\fq$-rational lines $PQ$ and $PR$.
\item $\ord(\sigma)\mid(q^2+q+1)$ and $\sigma$ stabilizes three non-collinear $\fqt\setminus\fq$-rational points $P,Q,R$ which are an orbit of the $\fq$-Frobenius collineation;
$\sigma$ stabilizes the $\fqt\setminus\fq$-rational lines $PQ$, $PR$, $QR$.
\end{itemize}
\end{remark}

Remark \ref{lemma:transitivity} lists some of the
 actions of $\PGL(3,q)$, which will be used in the proof of Theorem \ref{th:mobius}.


\begin{remark}\label{lemma:transitivity}
Let $q$ be a prime power and $G=\PGL(3,q)$.
\begin{itemize}
\item $G$ is $2$-transitive on the points of $\PG(2,q)$.
\item $G$ is transitive on the points of $\PG(2,q^2)\setminus\PG(2,q)$; the stabilizer in $G$ of a point $P\in\PG(2,q^2)\setminus\PG(2,q)$ stabilizes also its Frobenius conjugate $P^q$.
\item $G$ has two orbits on the points of $\PG(2,q^3)\setminus\PG(2,q)$; namely, one is made by the
points on the $\fq$-rational lines, the other is made by the remaining points.
\item $G$ is $2$-transitive on the $\fq$-rational lines.
\item $G$ is transitive on the $\fq$-rational point-line pairs $(P,\ell)$ with $P\in\ell$.
\item $G$ is transitive on the $\fq$-rational point-line pairs $(P,\ell)$ with $P\notin\ell$.
\item $G$ is transitive on the non-collinear triples $(P,Q,R)$ of $\fq$-rational points.
\item $G$ is transitive on the $\fqt\setminus\fq$-rational triangles $\{P,P^q,P^{q^2}\}\subset\PG(2,q^3)\setminus\PG(2,q)$ left invariant by the $\fq$-Frobenius collineation.
\item $G$ is transitive on the projective frames of $\PG(2,q)$, i.e. on the $4$-tuples of $\fq$-rational points no three of which are collinear.
\end{itemize}
\end{remark}

\section{The M\"obius function of $\PSL(3,2^p)$ for any odd prime $p$}\label{sec:mainresult}

We state the main result, Theorem \ref{th:mobius}, whose proof is worked out in Section \ref{sec:proof}.
We assume that  $p$ is an odd prime and $q=2^p$,
so that $G=\PSL(3,q)=\PGL(3,q)$.
The main argument in the proof is to find which subgroups of $G$ are intersection of maximal subgroups.
Roughly speaking,  we start with the intersection of two
maximal subgroups $M_1$ and $M_2$. Through  the geometry of $M_1$
and $M_2$, we determine the structure of $M_1\cap
M_2$ and we are able to identify which other maximal subgroups
of $G$ contain $M_1\cap M_2$ ; also, we study whether the group
$M_1\cap M_2$ is unique up to conjugation in $G$. Clearly, the group
$M_1\cap M_2$ may vary when $M_1$ and $M_2$ run in their conjugacy
classes. For instance, if $M_1$ is the stabilizer of a point
$P\in\PG(2,q)$ and $M_2$ is the stabilizer of an $\fq$-rational
triangle $T$, then $M_1\cap M_2\cong(C_{q-1})^2: C_2$ if $P$ is a
vertex of $T$; $M_1\cap M_2\cong C_{2(q-1)}$ if $P$ is on a side of
$T$ but not a vertex; and $M_1\cap M_2\cong \Sym(3)$ if $P$ is not
on a side of $T$.
We continue by intersecting $M_1\cap M_2$ with other maximal subgroups, stopping when the geometry of the chosen maximal subgroups forces their intersection to be trivial.

\begin{theorem}\label{th:mobius}
The subgroups $H<G$ which are intersection of maximal subgroups of
$G$ are exactly the groups in Table \ref{tabella:teorema}, where the
normalizer $N_G(H)$ and the M\"obius function $\mu(H)$ are provided.
For any such $H$ there is just one conjugacy class in $G$.
\end{theorem}

\section{Proof of Theorem \ref{th:mobius}}\label{sec:proof}

We use the same notation as in Section \ref{sec:mainresult}. The
proof is divided into the following
steps: for any $H$ in Table
\ref{tabella:teorema}, we prove that $G$ has exactly one conjugacy class, and we determine $N_G(H)$ (Proposition \ref{prop:conjugatesandnormalizers}); we show that the intersections of maximal subgroups of $G$ are exactly the groups in Table \ref{tabella:teorema}
(Propositions \ref{prop:isintersection} and \ref{prop:ismaxint}); for any $H$ in Table \ref{tabella:teorema}, we determine $\mu(H)$ (Proposition \ref{prop:mobius}).

\begin{proposition}\label{prop:conjugatesandnormalizers}
For any group $H$ in Table \ref{tabella:teorema} there is exactly
one conjugacy class in $G$, and $N_G(H)$ is as in Table
\ref{tabella:teorema}.
\end{proposition}

\begin{proof}
We consider the groups $H$ according to their line in Table
\ref{tabella:teorema}.
\begin{itemize}
\item {\bf Lines 1 to 5.} $H$ is a maximal subgroup; the claim follows from Theorem \ref{th:psl3q}.

\item {\bf Line 31, $H=\{1\}$.} The claim is trivial.

\item {\bf Line 30, $H\cong C_2$.} The involution $\alpha$ of $H$ is an elation, and hence is uniquely determined by its center $P$, its axis $\ell$, and its action on a third point not on $\ell$. Hence, there is just one conjugacy class for $H$ by Lemma \ref{lemma:transitivity}.
Also, $N_G(H)$ stabilizes $P$ and $\ell$. Thus, up to conjugation,
$P=(1:0:0)$, $\ell:Z=0$, and
\begin{footnotesize}
\begin{equation}\label{eq:normalizerinvolution}
 \alpha = \begin{pmatrix} 1 & 0 & 1 \\ 0 & 1 & 0 \\ 0 & 0 & 1 \end{pmatrix},\ N_G(H) = \left\{\sigma_{a,b,c,\lambda}= \begin{pmatrix} 1 & a & b \\ 0 & \lambda & c \\ 0 & 0 & 1 \end{pmatrix} : a,b,c\in\fq,\lambda\in\fq^* \right\}\cong E_q^{1+2}: C_{q-1}.
 \end{equation}
\end{footnotesize}


\item
{\bf Line 29, $H\cong C_3$.} Since $3\mid(q+1)$, the fixed points of
a generator $\alpha$ of $H$ are three non-collinear points
$P\in\PG(2,q)$, $Q,R\in\PG(2,q^2)\setminus\PG(2,q)$; note that
$R=Q^q$ and the line $\ell=QR$ is $\fq$-rational. By Lemma
\ref{lemma:transitivity} $P$ and $\ell$ are unique up to
conjugation; since $G_{P,\ell}$ acts as $\GL(2,q)$ on $\ell$,
$\{Q,R\}$ is also unique up to conjugation. Hence, there is just one
conjugacy class for $H$.
The pointwise stabilizer of $\{P,Q,R\}$ in $G$ is cyclic of order
$q^2-1$. Also, $N_G(H)$ stabilizes $P$ and acts on $\{Q,R\}$. Thus,
$|N_G(H)|=2(q^2-1)$ and $N_G(H)\cong C_{q^2-1}:C_2$.



\item
{\bf Line 28, $H\cong E_4\leq G_{P_1,P_2,P_3}$.} Here, $P_1,P_2,P_3$
are three collinear $\fq$-rational points. For $i\in\{1,2,3\}$ let
$\alpha_i\in G$ be the elation with center $P_i$ and axis $\ell=P_1
P_2 P_3$, with $\alpha_3=\alpha_1\alpha_2$ and
$H=\langle\alpha_1,\alpha_2\rangle\cong E_4$. For any
$P\in\PG(2,q)\setminus\ell$, the set $F=\{P_1,P_2,P,\alpha_3(P)\}$
is a projective frame of $\PG(2,q)$. Also, $H$ is uniquely
determined by $F$; in fact, $\alpha_1(P)=P P_1 \cap P_2\alpha_3(P)$
and $\alpha_2(P)=P P_2 \cap P_1\alpha_3(P)$. Then there is just one
conjugacy class for $H$ by Lemma \ref{lemma:transitivity}.

The normalizer $N_G(H)$ acts on $\{P_1,P_2,P_3\}$; for any
$\sigma\in N_G(H)$, $\sigma(P_i)=P_j$ implies
$\sigma\alpha_i\sigma^{-1}=\alpha_j$. Thus, the pointwise stabilizer
$S$ of $\{P_1,P_2,P_3\}$ in $N_G(H)$ is made by those $\sigma\in
N_G(H)$ which commute with $\alpha_i$ for any $i$ and stabilize
$\ell$ pointwise. Since no homology with axis $\ell$ commutes with
an elation with axis $\ell$, $S$ is only made by elations with axis
$\ell$. Since two elations commute if and only if they have the same
center or axis, this implies $S\cong E_{q^2}$. Now, $N_G(H)/E_{q^2}$
acts faithfully on $\{P_1,P_2,P_3\}$ and hence is a subgroup of
$\Sym(3)$. Finally we can choose $H$ up to conjugation and obtain
$N_G(H)$ as follows:
\begin{footnotesize}
$$
\alpha_1=\begin{pmatrix} 1 & 0 & 1 \\ 0 & 1 & 0 \\ 0 & 0 & 1
\end{pmatrix},\quad \alpha_2=\begin{pmatrix} 1 & 0 & 0 \\ 0 & 1 & 1
\\ 0 & 0 & 1 \end{pmatrix}, \quad \sigma_{0,b,c,1}=\begin{pmatrix} 1
& 0 & b \\ 0 & 1 & c \\ 0 & 0 & 1 \end{pmatrix}, \quad
\tau=\begin{pmatrix} 0 & 1 & 0 \\ 1 & 1 & 0 \\ 0 & 0 & 1
\end{pmatrix}, \quad \omega=\begin{pmatrix} 0 & 1 & 0 \\ 1 & 0 & 0
\\ 0 & 0 & 1 \end{pmatrix}.
$$
\end{footnotesize}
By direct computation,
$N_G(H)=\{\sigma_{0,b,c,1}:b,c\in\fq\}:\langle\tau,\omega\rangle\cong
E_{q^2}:\Sym(3)$.

\item
{\bf Line 27, $H\cong E_4\leq G_{\ell_1,\ell_2,\ell_3}$.} Here,
$\ell_1,\ell_2,\ell_3$ are three $\fq$-rational lines, concurrent in
$P$. Let $H=\{1,\alpha_1,\alpha_2,\alpha_3=\alpha_1\alpha_2\}\cong
E_4$, where $\alpha_i$ has center $P$ and axis $\ell_i$. A dual argument with respect to the one used in the previous point, shows that there is just one conjugacy class for
$H$, and $N_G(H)\cong E_{q^2}: \Sym(3)$.


\item
{\bf Line 26, $H\cong C_4$.} Let $\alpha$ be a generator of $H\cong
C_4$, with fixed point $P\in\PG(2,q)$ and fixed $\fq$-rational line
$\ell$, where $P\in\ell$. Then $\alpha$ is uniquely determined by
its action on $F=\{P,Q,\alpha(Q),R\}$, where $Q\in\PG(2,q)\setminus\ell$ and $R\in\PG(2,q)\setminus \overline{Q\alpha(Q)}$. Since $F$ is a projective frame of
$\PG(2,q)$, there is just one conjugacy class for $H$ in $G$. Up to
conjugacy, we have
\begin{footnotesize}
\begin{equation}\label{eq:C4}
 \alpha=\begin{pmatrix} 1 & 1 & 1 \\ 0 & 1 & 1 \\ 0 & 0 & 1 \end{pmatrix},\quad \alpha^2=\begin{pmatrix} 1 & 0 & 1 \\ 0 & 1 & 0 \\ 0 & 0 & 1 \end{pmatrix}.
 \end{equation}
\end{footnotesize}
Thus, $N_G(H)\leq N_G(\langle\alpha^2\rangle)$ and
$N_G(\langle\alpha^2\rangle)$ is in Equation
\eqref{eq:normalizerinvolution}. By direct checking,
$\sigma_{a,b,c,\lambda}\in N_G(H)$ if and only if $\lambda=1$ and
either $a=c+1$ or $a=c$. Also,
$Z(N_G(H))=\{\sigma_{0,b,c,1}:b\in\fq\}$. Therefore, $N_G(H)\cong
E_q\,.\,E_{2q}$.


\item
{\bf Line 25}, $H\cong \Sym(3)$. Let
$H=\langle\alpha\rangle:\langle\beta\rangle$ with $o(\alpha)=3$,
$o(\beta)=2$. Then $\langle\alpha\rangle$ is uniquely determined by
its fixed points $P\in\PG(2,q)$,
$Q,R\in\PG(2,q^2)\setminus\PG(2,q)$, while $\beta$ fixes $P$,
interchanges $Q$ ad $R$, and is uniquely determined by the
projective frame $F=\{P,Q,R,S\}$, where $S$ is an $\fq$-rational
point on the axis of $\beta$ different from $P$. Hence, there is
just one conjugacy class for $H$ in $G$.

From the previous points about $C_2$ and $C_3$ follows that
$N_G(H)$ contains $H\times C_{q-1}$, where $C_{q-1}$ is made by the
homologies with center $P$ and axis $QR$. Also, $N_G(H)$ fixes $P$
and acts on the three intersection points between $QR$ and the axes
of the three elations in $H$. Since $C_{q-1}$ is the whole subgroup
of $G_P$ acting trivially on $QR$, $N_G(H)/C_{q-1}\leq
\Sym(3)$ and hence $N_G(H)=H\times C_{q-1}$.

\item
{\bf Line 24, $H\cong C_7\leq G_{T,\Pi}$.} Here, $\Pi$ is a subplane of order $2$, and $T$ is a triangle; we have
$G_T\cong C_{q^2+q+1}: C_3$ or $G_T\cong(C_{q-1})^2: \Sym(3)$,
according to $p>3$ or $p=3$, respectively. Suppose $p>3$.  There is
just one conjugacy class for $H$ in $G$, because $H$ is
characteristic in $G_T$ which is unique up to conjugation in $G$ (Theorem \ref{th:psl3q}); this also shows $N_G(H)=G_T\cong C_{q^2+q+1}:C_3$. If $p=3$, the claim follows by direct inspection with {\sc Magma} \cite{MAGMA}.



\item
{\bf Line 23, $H\cong D_8$.} As in Line 26, we can assume that an element $\alpha\in H$ of order $4$ is
as in Equation \eqref{eq:C4}. Let $\beta\in H$ be an involution
with $\beta\ne\alpha^2$. Let $\tau\in N_G(H)\leq
N_G(\langle\alpha\rangle)$ be an involution with
$\tau\alpha\tau^{-1}=\alpha^{-1}$. By direct checking, either
$\tau=\sigma_{1,b,0,1}$ or $\tau=\sigma_{0,b,1,1}$ for some
$b\in\fq$; also, $\tau$ is conjugated by some $\sigma_{0,b,0,1}$
either to $\beta$ or to $\alpha\beta$. Thus, there is just one
conjugacy class for $H$ in $G$. By direct checking, an element
$\sigma_{a,b,c,1}\in N_G(\langle\alpha\rangle)$ is in $N_G(H)$ if
and only if $a,c\in\{0,1\}$, and
$Z(N_G(H))=\{\sigma_{0,b,0,1}:b\in\fq\}$. Therefore, $N_G(H)\cong
E_q\,.\,E_4$.



\item
{\bf Line 22, $H\cong C_7:C_3$.} If $p=3$, then the claim follows using {\sc Magma} \cite{MAGMA}. Suppose $p>3$. Then
$H$ is characteristic in $M\cong C_{q^2+q+1}:C_3$, which is unique up to conjugation; hence, $H$ is unique up to conjugation. $N_G(H)$
stabilizes the triangle $\tilde{T}$ fixed pointwise by $C_7\leq H$,
hence $N_G(H)\leq M$; thus $N_G(H)\cong C_m:C_3$ with
$7\mid m\mid(q^2+q+1)$. Since $G$ does not contain subgroup $E_9$ (as ${\rm Syl}_3(G)$ is cyclic), the action by conjugation of the $2m$ $3$-elements of $N_G(H)$ on the $14$ $3$-elements of $H$ is fixed-point-free;
thus, $m=7$ and $N_G(H)=H$.



\item
{\bf Line 21, $H\cong \Sym(4)= G_{\ell,\Pi}$}. For any subplane
$\Pi$ of order $2$, the maximal subgroup
$G_{\Pi}\cong\PSL(3,2)$ of $G$ has two conjugacy classes of
subgroups $\Sym(4)$, containing respectively the groups
$G_{\ell,\Pi}$ and the groups $G_{P,\Pi}$, where $\ell$ and $P$ range over the $7$ lines and points of $\Pi$.
As $G_{\Pi}$ is unique up to conjugation in $G$, the same holds for $H$.
For any $\sigma$ in the centralizer $C_G(H)$, $\sigma$ commutes with all elations in $H$ and hence stabilizes
their centers; hence, $\sigma$ stabilizes $\Pi$ pointwise, so that $C_G(H)$ is trivial. Thus $N_G(H)=H$, because ${\rm
Aut}(\Sym(4))\cong \Sym(4)$.

\item {\bf Line 20, $H\cong\Sym(4)= G_{P,\Pi}$.} The claim follows as in the previous point.

\item
{\bf Line 19, $H\cong C_{q-1}=
G_{P_1,\ldots,P_{q+1},\ell_1,\ldots,\ell_{q+1}}$.} Here,
$P_1,\ldots,P_{q+1}\in\PG(2,q)$ are distinct collinear points and
$\ell_1,\ldots,\ell_{q+1}$ are $\fq$-rational distinct lines,
concurrent in a point $P$ and different from $\ell=P_1P_2\cdots
P_{q+1}$. It is easily seen that there
is just one conjugacy class for $H$ in $G$, determined by $(P,\ell)$;
$H$ is the center of $N_G(H)$,
because the following holds up to conjugation:
\begin{footnotesize}
\begin{equation}\label{eq:homologies}
P=(1:0:0),\ \ell:X=0,\ N_G(H)=\left\{ \begin{pmatrix} 1 & 0 & 0
\\ 0 & a & b \\ 0 & c & d \end{pmatrix} : a,b,c,d\in\fq,\,ad-bc\ne0
\right\}\cong\GL(2,q).
\end{equation}
\end{footnotesize}



\item
{\bf Line 18, $H\cong E_q=
G_{P_1,\ldots,P_{q+1},\ell_1,\ldots,\ell_{q+1}}$.} Here,
$P_1,\ldots,P_{q+1}\in\PG(2,q)$ are distinct points, collinear in a
line $\ell$, and $\ell_1,\ldots,\ell_{q+1}$ are $\fq$-rational
distinct lines, concurrent in a point $P$ of $\ell$. Then $H$ is the
group of elations with center $P$ and axis $\ell$; $H$ is uniquely
determined by $(P,\ell)$. Hence, there is just one conjugacy class
for $H$ in $G$. Also, $N_G(H)$ fixes $P$ and stabilizes $\ell$
linewise. Thus, up to conjugation and by direct checking, $H$ and
$N_G(H)$ are as follows:
\begin{footnotesize}
\begin{equation}\label{eq:elations}
H=\left\{\begin{pmatrix}1 & 0 & b \\ 0 & 1 & 0 \\ 0 & 0 &
1\end{pmatrix}\right\}_{b\in\fq},\ N_G(H)=
\left\{\begin{pmatrix}\lambda & a & b \\ 0 & \mu & c \\ 0 & 0 & 1
\end{pmatrix}:\lambda,\mu,a,b,c\in\fq,\, \lambda,\mu\ne0\right\}
\cong E_q^{1+2}:(C_{q-1})^2.
\end{equation}
\end{footnotesize}

\item
{\bf Line 17, $H\cong C_{2(q-1)}$.} By Remark \ref{remark:action},
$H=\langle\alpha\rangle$ stabilizes two distinct points
$P,C\in\PG(2,q)$, the line $CP$, and another $\fq$-rational line
$\ell$ through $P$. By Line 19,
$\langle\alpha^2\rangle$ is unique up to conjugation in $G$. The involutions of $N_G(\langle\alpha^2\rangle)$ form a unique
conjugacy class, as the same happens in
$N_G(\langle\alpha^2\rangle)/\langle\alpha^2\rangle\cong\PGL(2,q)$.
Hence, there is just one conjugacy class for $H$ in $G$. Since the
normalizer $N_G(H)$ stabilizes $\{P,C,CP,\ell\}$ elementwise, we
have up to conjugation and by direct computation that $P=(1:0:0)$,
$C=(0:1:0)$, $CP:Z=0$, $\ell:Y=0$, and
\begin{equation}
H=\left\langle \begin{pmatrix} 1 & 0 & 1 \\ 0 & \epsilon & 0 \\ 0 &
0 & 1 \end{pmatrix} \right\rangle,\quad N_G(H)=\left\{
\begin{pmatrix} 1 & 0 & d \\ 0 & \mu & 0 \\ 0 & 0 & 1 \end{pmatrix}
: d\in\fq,\mu\in\mathbb{F}_q^* \right\}\cong E_q\times C_{q-1},
\end{equation}
where $\epsilon$ is a primitive element of $\fq$.


\item
{\bf Line 16, $H\cong E_q:C_{q-1}=G_{\ell_1,\ldots,\ell_{q+1},P}$.}
Here $P\in\PG(2,q)$ and $\ell_1,\ldots,\ell_{q+1}$ are
$\fq$-rational lines concurrent in a point $C\ne P$. The elementwise
stabilizer of $\{\ell_1,\ldots,\ell_{q+1},P\}$ contains the
elations $E_q$ of center $C$ and axis $CP$, and the homologies with
center $C$ and axis through $P$; such homologies form
subgroups $C_{q-1}$ which are conjugated under $E_q$, as $E_q$
acts regularly on the $q$ $\fq$-rational lines through $P$ different
from $CP$. Thus, $H=G_{\ell_1,\ldots,\ell_{q+1},P}\cong E_q:
C_{q-1}$ and no elation and homology in $H$ commute.
As $G$ is $2$-transitive on $\PG(2,q)$, $H$ is unique up to conjugation in $G$. Also, $N_G(H)$ stabilizes both $P$ and $C$; hence $|N_G(H)|$ divides $q^2(q-1)^2$ by the
orbit-stabilizer theorem. Up to conjugacy and by direct checking, we
have $C=(1:0:0)$, $P=(0:1:0)$, and
\begin{footnotesize}
\begin{equation}\label{eq:EqCq-1}
    H=\left\{ \begin{pmatrix} \lambda & 0 & b \\ 0 & 1 & 0 \\ 0 & 0 & 1 \end{pmatrix} \right\}_{c\in\fq,\lambda\in\mathbb{F}_q^*},\
    N_G(H)=\left\{\begin{pmatrix} \lambda & 0 & b \\ 0 & \mu & c \\ 0 & 0 & 1 \end{pmatrix}:\lambda,\mu,b,c\in\fq,\, \lambda,\mu\ne0\right\}\cong E_{q^2}:(C_{q-1})^2.
\end{equation}
\end{footnotesize}

\item
{\bf Line 15, $H\cong E_q:C_{q-1}=G_{P_1,\ldots,P_{q+1},\ell}$.}
Here $P_1,\ldots,P_{q+1}$ are $\fq$-rational points collinear in a
line $r$, and $\ell\ne r$ is another $\fq$-rational line. A dual argument with respect to the one in the previous point yields the claim.


\item
{\bf Line 14, $H\cong(C_{q-1})^2$.} $H$ is the pointwise
stabilizer of an $\fq$-rational non-degenerate triangle $T$, and
hence is unique up to conjugation in $G$. $N_G(H)$ is
the stabilizer of $T$, i.e. a maximal subgroup $(C_{q-1})^2:\Sym(3)$
of $G$.


\item
{\bf Line 13, $H\cong (C_{q-1})^2:C_2$.} The characteristic subgroup
$(C_{q-1})^2$ of $H$ is unique up to conjugation in $G$, as shown in
the previous point. By Schur-Zassenhaus theorem, the complement
$C_2$ is unique up to conjugation in $H$. Hence, there is just one
conjugacy class for $H$ in $G$. Let $T$ be the triangle pointwise
stabilized by $(C_{q-1})^2$ and $P$ be the vertex of $T$ stabilized
by $H$. Then $N_G(H)$ stabilizes $T$, so that $N_G(H)\leq G_T$.
Also, $N_G(H)$ stabilizes $P$, which implies $N_G(H)=H$.


\item
{\bf Line 12, $H\cong E_q:(C_{q-1})^2$.} Such a $H\leq G$
exists, for instance as follows:
\begin{equation}\label{eq:EqCq-1Cq-1}
\left\{ \begin{pmatrix} \lambda & 0 & c \\ 0 & \mu & 0 \\ 0 & 0 & 1
\end{pmatrix}:c,\lambda,\mu\in\fq,\, \lambda,\mu\ne0 \right\}.
\end{equation}
Consider a group $L:M\leq G$, with $L\cong E_q$, $M\cong(C_{q-1})^2$. The group $L$ is made by elations with either the same center or the same axis.
Suppose that the elations of $L$ have the same axis $\ell$. Consider
an element $\beta\in M$ of order $q-1$ with exactly three fixed
points $P,Q,R$; let $R$ be the one out of $\ell$. Then $\beta$ does
not commute with any $\alpha\in L^*$ (otherwise $\alpha$ stabilizes
$R$), and hence the action of
$\langle\beta\rangle$ by conjugation on $E_q^*$ is fixed-point-free.
Also, $\beta\alpha\beta^{-1}$ has the same center as $\alpha$.
Therefore, the elations of $L$ have the same center. Similarly,
if the elations of $L$ have the same center, they also have the same axis. Thus, $L$ is made by
the elations with the same center $P$ and axis $\ell$.
Being uniquely determined by $(P,Q,R)$, $H$ is unique up to conjugation in $G$. Being $L$ characteristic in $H$, we have $N_G(H)\leq N_G(L)$; thus, by Equations \eqref{eq:elations} and \eqref{eq:EqCq-1Cq-1}, $N_G(H)=H$.

\item
{\bf Line 11, $H\cong
E_{q^2}:C_{q-1}=G_{\ell_1,\ldots,\ell_{q+1}}$.} Here,
$\ell_1,\ldots,\ell_{q+1}$ are distinct $\fq$-rational lines
concurrent in a point $P$. Thus, the elementwise stabilizer of
$\{\ell_1,\ldots,\ell_{q+1}\}$ is made by the elations and
homologies with center $P$. The elations with center $P$ (resp. the homologies with center $P$ and given axis $\ell$) form a
subgroup $L\cong E_{q^2}$ (resp. $M\cong C_{q-1})$. The action by
conjugation of the elements of $L$ on the homologies with center $P$
is fixed-point-free (as no elation with center $P$ stabilizes a line
$\ell$ not through $P$); hence, $G_{\ell_1,\ldots,\ell_{q+1}}=L:M$.
Since $H$ is uniquely determined by $P$ and $G$ is transitive on
$\PG(2,q)$, $H$ is unique up to conjugation in $G$. As $N_G(H)=G_P$, Theorem \ref{th:psl3q} yields the claim.



\item
{\bf Line 10, $H\cong E_{q^2}:C_{q-1}=G_{P_1,\ldots,P_{q+1}}$.}
Here, $P_1,\ldots,P_{q+1}$ are distinct collinear $\fq$-rational
points. The claim follows with a dual argument with respect to the
one used in the previous point.


\item
{\bf Line 9, $H\cong E_{q^2}:(C_{q-1})^2=G_{\ell,r}$.} Here, $\ell,r$ are distinct $\fq$-rational lines. Since $G$ is
$2$-transitive on $\PG(2,q)$ and hence also in the dual plane, $G_{\ell,r}$ is unique up to conjugation in $G$. Let
$P=\ell\cap r$, so that $G_{\ell,r}\leq G_P$. Then $G_{\ell,r}$
contains the group $L\cong E_{q^2}$ of elations with center $P$, and
by Remark \ref{remark:action} $G_{\ell,r}$ does not contain any
other $2$-element; in particular, $L\trianglelefteq G_{\ell,r}$.
Also, $G_{\ell,r}$ contains the pointwise stabilizer
$N\cong(C_{q-1})^2$ of an $\fq$-rational triangle with two sides in
$\ell$ and $r$. Hence, $G_{\ell,r}$ contains $L:N$. Every nontrivial
element whose order divides $q+1$ does not stabilize two
$\fq$-rational lines; since $|G_P|=q^3(q+1)(q-1)^2$, this implies
$G_{\ell,r}=L:N\cong E_{q^2}:(C_{q-1})^2$. $N_G(H)$ acts on $\{\ell,r\}$, and $H$ is the pointwise stabilizer of $\ell,r$. Thus, $N_G(H)\cong H:C_2$, where the complement $C_2$ is given by a suitable elation; in fact, the
group of elations with center $C\ne P$ and axis $CP$ acts
transitively on the lines through $P$ different from $CP$.



\item
{\bf Line 8, $H\cong E_{q^2}:(C_{q-1})^2=G_{P,Q}$.} Here,
$P,Q\in\PG(2,q)$, $P\ne Q$. The claim follows with a dual
argument with respect to the one used in the previous point.


\item
{\bf Line 7, $H\cong\GL(2,q)$.} Such a subgroup $H\leq G$ exists, as
shown in Equation \eqref{eq:homologies}, and $H=N_G(Z)$, where
$Z\cong C_{q-1}$ is the center of $H$ and is made by homologies with
a given center $P$ and axis $\ell$. $H$ is uniquely
determined by $(P,\ell)$; hence, by Remark \ref{remark:action}, $H$
is unique up to conjugation in $G$. Since $Z$ is characteristic in
$H$, we have $N_G(H)\leq N_G(Z)=H$ and the claim follows.



\item
{\bf Line 6, $H\cong E_q^{1+2}:(C_{q-1})^2$.} Equation
\eqref{eq:elations} shows that such a subgroup $H\leq G$ exists, and
$H=N_G(K)$ where $K\cong E_q$ is made by the elations with a common
center $P$ and axis $\ell$. Since $K$ is unique up to conjugation in
$G$ (see the point above regarding Line 18), $H$ is also unique up
to conjugation in $G$. $N_G(H)$ stabilizes $P$ and
$\ell$, and hence preserves the elations with center $P$ and axis
$\ell$. This implies $N_G(H)\leq N_G(K)$, so $N_G(H)=H$.
\end{itemize}
\end{proof}

\begin{proposition}\label{prop:isintersection}
Every $H$ in Table \ref{tabella:teorema} is intersection
of maximal subgroups of $G$.
\end{proposition}

\begin{proof}
For every group $H$, we use what has been already shown
in the proof of Proposition \ref{prop:conjugatesandnormalizers}
and we give a summary description of each case.
\begin{itemize}
\item
{\bf Lines 1 to 5.} In these cases, $H$ is a maximal subgroup of
$G$.

\item {\bf Lines 6 and 7}, $H\cong E_q^{1+2}:(C_{q-1})^2$ and $H\cong\GL(2,q)$, respectively.
We have $H\leq G_{P,\ell}$, where $P$ and $\ell$ are an
$\fq$-rational point and line, with either $P\in\ell$ (Line 6) or
$P\notin\ell$ (Line7). From Remark \ref{remark:action} and the
orbit-stabilizer theorem follows $H=G_{P,\ell}=G_P\cap G_{\ell}$. No
other maximal subgroup contains $H$.

\item
{\bf Line 8}, $H=G_{P,Q}\cong E_{q^2}:(C_{q-1})^2=
G_P\cap G_Q$ for some distinct points $P,Q\in\PG(2,q)$. No other
maximal subgroup  contains $H$.




\item
{\bf Line 9}, $H=G_{\ell,r}\cong E_{q^2}:(C_{q-1})^2=G_{\ell}\cap G_r$ for some distinct $\fq$-rational lines
$\ell,r$. No other maximal subgroup contains $H$.

\item
{\bf Line 10, $H=G_{P_1,\ldots,P_{q+1}}\cong E_{q^2}:C_{q-1}$.} We
have $H=G_{P_1}\cap\cdots\cap G_{P_{q+1}}\cap G_{\ell}$, where
$\ell$ is an $\fq$-rational line and $P_1,\ldots,P_{q+1}$ are the
$q+1$ distinct $\fq$-rational points of $\ell$. No other maximal
subgroup of contains $H$.

\item
{\bf Line 11, $H=G_{\ell_1,\ldots,\ell_{q+1}}\cong
E_{q^2}:C_{q-1}$.} We have $H=G_P\cap G_{\ell_1}\cap\cdots\cap
G_{\ell_{q+1}}$, where $P\in\PG(2,q)$ and $\ell_1,\ldots,\ell_{q+1}$
are the $q+1$ distinct $\fq$-rational lines through $P$. No other
maximal subgroup contains $H$.

\item
{\bf Line 12, $H\cong E_q:(C_{q-1})^2$.} Up to conjugation, $H$ is
as in Equation \eqref{eq:EqCq-1Cq-1}. Then $H=G_{P,Q,\ell,r}$, where $P=(1:0:0)$, $Q=(0:1:0)$, $\ell:Y=0$, $r:Z=0$. Thus $H=G_P\cap G_Q\cap G_{\ell}\cap G_r$. No other maximal subgroup contains $H$.

\item
{\bf Line 13, $H\cong(C_{q-1})^2 : C_2$.} $H$ stabilizes an
$\fq$-rat. triangle $T=\{P,Q,R\}$ and a vertex $P$. Then
$H=G_P\cap G_{QR}\cap G_T$. No other max. subgroup
contains $H$.


\item
{\bf Line 14, $H\cong(C_{q-1})^2$.} $H$ is the pointwise
stabilizer in $G$ of an $\fq$-rational triangle $T=\{P,Q,R\}$. Then
$H=G_P\cap G_Q\cap G_R\cap G_{PQ}\cap G_{PR}\cap G_{QR}\cap G_T$,
and no other maximal subgroup contains $H$.

\item
{\bf Line 15, $H=G_{P_1,\ldots,P_{q+1},\ell}\cong E_q : C_{q-1}$.}
$H$ is the subgroup of $G$ stabilizing the line $r$
through $P_1,\ldots,P_{q+1}$ pointwise, and the line $\ell$
linewise. Hence, $H=G_{P_1}\cap\cdots \cap G_{P_{q+1}}\cap G_{r}\cap
G_{\ell}$. No other maximal subgroup contains $H$.


\item
{\bf Line 16, $H=G_{\ell_1,\ldots,\ell_{q+1},P}\cong E_q :
C_{q-1}$.} $H$ is the subgroup of $G$ stabilizing
all lines $\ell_1,\ldots,\ell_{q+1}$ through a point
$C\in\PG(2,q)$, and a point $P\in\PG(2,q)\setminus\{C\}$. Hence,
$H=G_P\cap G_C\cap G_{\ell_1}\cap\cdots \cap G_{\ell_{q+1}}$. No
other max. subgroup contains $H$.

\item
{\bf Line 17, $H\cong C_{2(q-1)}$.} We have $H\leq G_{P,C,CP,\ell}$,
where $P,C\in\PG(2,q)$, $P\ne C$, $\ell$ is an $\fq$-rat. line,
$P\in\ell$, $\ell\ne CP$. No other $\fq$-rat. point or line is
stabilized by $H$. Also, $H$ stabilizes exactly $q$
$\fq$-rat. triangles $T_i=\{C,Q_i,R_i\}$, where the $\{Q_i,R_i\}$'s are the
orbits of $H$ on $\ell(\fq)\setminus\{P\}$. Thus, $H = G_P\cap
G_C\cap G_{CP}\cap G_{\ell}\cap G_{T_1}\cap\cdots\cap G_{T_q}$, and
no other maximal subgroup contains $H$.


\item
{\bf Line 18, $H=
G_{P_1,\ldots,P_{q+1},\ell_1,\ldots,\ell_{q+1}}\cong E_q$.} Here,
$P_1,\ldots,P_{q+1}$ are the $\fq$-rat. points of a line $\ell$,
and $\ell_1,\ldots,\ell_{q+1}$ are the $\fq$-rat. lines through
a point $P$ of $\ell$; $H$ is the group of elations with
center $P$ and axis $\ell$. Thus, $H=G_{P_1}\cap\cdots\cap
G_{P_{q+1}}\cap G_{\ell_1}\cap\cdots\cap G_{\ell_{q+1}}$. No other
maximal subgroup contains $H$.


\item
{\bf Line 19,
$H=G_{P_1,\ldots,P_{q+1},\ell_1,\ldots,\ell_{q+1}}\cong C_{q-1}$.}
Here, $P_1,\ldots,P_{q+1}$ are the $\fq$-rational points of a line
$\ell$, and $\ell_1,\ldots,\ell_{q+1}$ are the $\fq$-rational lines
through a point $P\notin\ell$; $H$ is the group of homologies with
center $P$ and axis $\ell$. Also, $H$ stabilizes (pointwise) the
$\binom{q+1}{2}$ $\fq$-rational triangles with one vertex in $P$ and
two vertices on $\ell$. Thus, $H=G_P\cap G_{P_1}\cap\cdots\cap
G_{P_{q+1}}\cap G_{\ell_1}\cap\cdots\cap G_{\ell_{q+1}}\cap
G_{\ell}\cap G_{T_1}\cap\cdots\cap G_{T_{\binom{q+1}{2}}}$. No other
maximal subgroup contains $H$.


\item
{\bf Line 20, $H=G_{P,\Pi}\cong \Sym(4)$.} We have $H=G_P\cap G_\Pi$, for some subplane $\Pi$ of order $2$ and some point
$P\in\Pi$.
No other maximal subgroup contains $H$.



\item
{\bf Line 21, $H=G_{\ell,\Pi}\cong\Sym(4)$.} We have $H=G_P\cap
G_\Pi$ for some subplane $\Pi\subset\PG(2,q)$ of order $2$ and some
$\fq$-rational line $\ell$ containing three points of $\Pi$. No
other maximal subgroup contains $H$.


\item
{\bf Line 22, $H\cong C_7 : C_3$.} By double counting arguments, $H= G_{\tilde{T}}\cap G_\Pi$, with a subplane
$\Pi$ of order $2$ and a triangle $\tilde T$.
No other max. subgroup contains $H$.

\item
{\bf Line 23, $H\cong D_8$.} Let $P$ and $\ell$ be the unique point
and line stabilized by an element of order $4$ in $H$; then
$H\leq G_{P,\ell}$. By double counting arguments, $H$
is contained in exactly $\frac{q}{2}$ max. subgroups
$G_{\Pi_1},\ldots,G_{\Pi_{q/2}}$ isomorphic to $\PSL(3,2)$;
also, $H$ is equal to the intersection of any two of them (see Lines 20 and 21).
Thus, $H=G_P\cap G_{\ell}\cap G_{\Pi_1}\cap\cdots\cap G_{\Pi_{q/2}}$.
No other max subgroup contains $H$.


\item
{\bf Line 24, $H=C_7\leq G_{T,\Pi}$.} If $p>3$, then $H$ stabilizes
no $\fq$-rat. points or lines; if $p=3$, the $\fq$-rat.
points and lines stabilized by $H$ are the vertices $P,Q,R$ and the
sides of $T$.
The number of subplanes $\Pi_i$ of order $2$
stabilized by $H$ is either $\frac{q^2+q+1}{7}$ or $7$, according to
$p>3$ or $p=3$, respectively; $H$ is equal to the intersection of
any two of them (see Line 22). Thus, either $
H=G_T \cap M_1\cap\cdots\cap M_{(q^2+q+1)/7}$ or $ H=G_{P}\cap
G_{Q}\cap G_{R}\cap G_{PQ}\cap G_{PR}\cap G_{QR}\cap G_{T}\cap
M_1\cap\cdots\cap M_7$, according to $p>3$ or $p=3$, respectively.
No other max. subgroup contains $H$.


\item
{\bf Line 25, $H\cong\Sym(3)$.} Let $P\in\PG(2,q)$ and
$Q,R\in\PG(2,q^2)\setminus\PG(2,q)$ be the fixed points of the
$3$-elements of $H$. Then $P,QR$ are the unique $\fq$-rat.
point and line fixed by $H$. By double counting arguments, $H$
stabilizes exactly $q-1$ $\fq$-rat. triangles $T_i$'s and $q-1$ subplanes $\Pi_i$'s of order $2$. For $i\ne j$, we have $H=G_{T_i}\cap G_{T_j}$ and $H=G_{\Pi_i}\cap G_{\Pi_j}$. Thus $H=G_P\cap G_{QR}\cap G_{T_1}\cap\cdots\cap G_{T_{q-1}}\cap M_1\cap\cdots\cap M_{q-1}$. No other max. subgroup contains $H$.


\item
{\bf Line 26, $H=C_4$.} The group $H$ stabilizes exactly one
$\fq$-rat. point $P$ and one $\fq$-rat. line $\ell$.
By double counting arguments, $H$ stabilizes exactly $\frac{q^2}{4}$
subplanes $\Pi_i$'s of order $2$. Every overgroup of $H$ isomorphic
to $D_8$ stabilizes exactly $\frac{q}{2}$ of such subplanes (see
Line 23); hence, $G_{\Pi_1,\ldots,\Pi_{q^2/4}}=H$. Thus,
$H=G_P\cap G_{\ell}\cap M_1\cap\cdots\cap M_{q^2/4}$. No other
max. subgroup contains $H$.


\item
{\bf Line 27, $H\cong E_4$, $H\leq G_{\ell_1,\ell_2,\ell_3}$.} Let
$P$ the center of the elations in $H$, and
$\ell_1,\ldots,\ell_{q+1}$ be the $\fq$-rat. lines through $P$.
Then $H$ stabilizes $P,\ell_1,\ldots,\ell_{q+1}$, and no other
$\fq$-rat. points or lines. Also, $H$
stabilizes exactly $\frac{q^2}{4}$ subplanes $\Pi_i$'s of order $2$,
and $G_{\Pi_1,\ldots,\Pi_{q^2/4}}=H$ (see Line 23). Thus, $H=G_P\cap
G_{\ell_1}\cap\cdots\cap G_{\ell_{q+1}}\cap G_{\Pi_1}\cap\cdots\cap
G_{\Pi_{q^2/4}}$, and no other max. subgroup contains $H$.

\item
{\bf Line 28, $H\cong E_4$, $H\leq G_{P_1,P_2,P_3}$.} As in Line 27, $H=G_{P_1}\cap\cdots\cap G_{P_{q+1}}\cap
G_{\ell}\cap G_{\Pi_1}\cap\cdots\cap G_{\Pi_{q^2/4}}$, where $\ell$ is a line through the $P_i$'s, and the $\Pi_i$'s are subplanes of order $2$. No other max. subgroup contains $H$.

\item
{\bf Line 29, $H\cong C_3$.} The group $H$ stabilizes exactly one
$\fq$-rat. point $P$ and one $\fq$-rat. line $\ell$. By
double counting arguments,
\begin{equation}\label{eq:C3}
H=G_P\cap G_\ell \cap G_{T_1}\cap\cdots\cap G_{T_{(q^2-1)/3}}\cap
G_{\tilde{T}_1}\cap\cdots\cap G_{\tilde{T}_{2(q^2-1)/3}}\cap
G_{\Pi_1}\cap\cdots\cap G_{\Pi_{(q^2-1)/3}},
\end{equation}
where the $T_i$'s are distinct $\fq$-rat. triangles, the
$\tilde{T}_i$'s are distinct $\mathbb{F}_{q^3}\setminus\fq$-rat. triangles, and the $\Pi_i$'s are distinct subplanes of $\PG(2,q)$ of
order $2$. No other max. subgroup contains $H$.


\item
{\bf Line 30, $H\cong C_2$.} Let $\alpha$ be the elation of $H$,
$P_1,\ldots,P_{q+1}$ be the $\fq$-rat. points of the axis of
$\alpha$, and $\ell_1,\ldots,\ell_{q+1}$ be the $\fq$-rat. lines
through the center of $\alpha$. By double counting argument,
\begin{footnotesize}
\begin{equation}\label{eq:C2}
H=G_{P_1}\cap\cdots\cap G_{P_{q+1}}\cap G_{\ell_1}\cap\cdots\cap
G_{\ell_{q+1}}\cap G_{T_1}\cap\cdots\cap G_{T_{q^3(q-1)/2}}\cap
G_{\Pi_1}\cap\cdots\cap G_{\Pi_{q^3(q-1)/8}},
\end{equation}
\end{footnotesize}
where the $T_i$'s are distinct $\fq$-rat. triangles and the
$\Pi_i$'s are distinct subplanes of $\PG(2,q)$ of order $2$. No
other max. subgroups contain $H$.

\item
{\bf Line 31, $H=\{1\}$.} As $G$ is simple, $H$ is the Frattini
subgroup of $G$.

\end{itemize}
\end{proof}

\begin{proposition}\label{prop:ismaxint}
Let $H<G$ be the intersection of maximal subgroups of $G$. Then $H$
is one of the groups in Table \ref{tabella:teorema}.
\end{proposition}

\begin{proof}
The claim is proved as follows: we consider every subgroup $K<G$ in
Table \ref{tabella:teorema}, starting from the maximal subgroups of
$G$; for any maximal subgroup $M$ of $G$ satisfying $K\not\leq
M$, we show that $H:=K\cap M$ is one of the groups in
Table \ref{tabella:teorema}.
\begin{itemize}
\item {\bf Line 1:} $K=G_P$ with $P\in\PG(2,q)$.
\begin{itemize}
\item If $H=K\cap G_Q$, $Q\in\PG(2,q)\setminus\{P\}$, then $H=G_{P,Q}\cong E_{q^2}:(C_{q-1})^2$.
\item If $H=K\cap G_{\ell}$, $\ell$ an $\fq$-rat. line, then $H\cong E_q^{1+2}:(C_{q-1})^2$ if $P\in\ell$, $H\cong\GL(2,q)$ if $P\notin\ell$.
\item If $H=K\cap G_T$, $T$ an $\fq$-rat. triangle, then either $P$ is a vertex of $T$, and $H\cong(C_{q-1})^2: C_2$; or $P$ is not a vertex but on a side of $T$, and $H\cong C_{2(q-1)}$; or $P$ is not on the sides of $T$, and $H\leq \Sym(3)$.
\item If $H=K\cap G_{\tilde T}$, $\tilde{T}$ an $\mathbb{F}_{q^3}\setminus\fq$-rat. triangle, $H\leq C_3$ by Lagrange's theorem.
\item If $H=K\cap G_{\Pi}$, $\Pi$ a subplane of order $2$, then either $P\notin\Pi$ and $H=\{1\}$; or $P\in\Pi$ and $H\cong \Sym(4)$.
\end{itemize}

\item {\bf Line 2:} $K=G_\ell$ for some $\fq$-rational line $\ell$.
\begin{itemize}
\item If $H=K\cap G_r$, $r\ne\ell$ an $\fq$-rat. line, then $H=G_{\ell,r}=E_{q^2}:(C_{q-1})^2$.
\item If $H=K\cap G_T$, $T$ an $\fq$-rat. triangle, then either $\ell$ is a side of $T$, and $H\cong(C_{q-1})^2: C_2$; or $\ell$ contains exactly one vertex of $T$, and $H\cong C_{2(q-1)}$; or $\ell$ does not contain any vertex of $T$, and $H\leq \Sym(3)$.
\item If $H=K\cap G_{\tilde T}$, $\tilde T$ an $\mathbb{F}_{q^3}\setminus\fq$-rat. triangle, then $H\leq C_3$.
\item If $H=K\cap G_{\Pi}$, $\Pi$ a subplane of order $2$, then either $\ell$ is not a line of $\Pi$ and $H=\{1\}$; or $\ell$ is a line of $\Pi$ and $H\cong \Sym(4)$.
\end{itemize}

\item {\bf Line 3:} $K=G_T$ for some $\fq$-rational triangle $T=\{A,B,C\}$.
\begin{itemize}
\item If $H=K\cap G_{T^\prime}$ for some $\fq$-rat. triangle $T^\prime=\{A^\prime,B^\prime,C^\prime\}\ne T$ and $H\not\leq \Sym(3)$, then some element of $H\setminus\{1\}$ stabilizes $T$ and $T^\prime$ pointwise. Hence, $T\cup T^\prime\subset\{P\}\cup\ell$ for some point $P$ and line $\ell$, so that $T$ and $T^\prime$ have a vertex in common.
If $T$ and $T^\prime$ have another vertex in common, then $H$ is the
group $C_{q-1}$ of homologies. Otherwise, $H\cong C_{2(q-1)}$.
\item If $H=K\cap G_{\tilde T}$, $\tilde T$ an  $\mathbb{F}_{q^3}\setminus\fq$-rational triangle, then $H\leq C_3$.
\item If $H=K\cap G_{\Pi}$, $\Pi$ a subplane of order $2$, and $H\not\leq\Sym(3)$, then $q=8$ and $H\cong C_7:C_3$ by direct checking with {\sc Magma} \cite{MAGMA}.
\end{itemize}

\item {\bf Line 4:} $K=G_{\tilde T}$ for some $\fqt\setminus\fq$-rational triangle $\tilde T$.
\begin{itemize}
\item If $H=K\cap G_{\tilde{T}^\prime}$, for some $\fqt\setminus\fq$-rational triangle $\tilde{T}^\prime \ne \tilde{T}$, then $H\leq C_3$.
\item If $H=K\cap G_{\Pi}$, $\Pi$ a subplane of order $2$, and $H\not\leq C_3$, then $H\cong C_7: C_3$.
\end{itemize}

\item {\bf Line 5:} $K=G_\Pi$ for some subplane $\Pi$ of order $2$. Let $H=K\cap G_{\Pi^\prime}$, $\Pi^\prime$ a subplane of order $2$.
Assume that $H$ contains the only proper subgroup of $\PSL(3,2)$ not appearing in Table \ref{tabella:teorema}, namely $\Alt(4)$. Let $\Lambda\subset\PG(2,q)$ be the pointset containing the centers of the $3$ elations in $\Sym(4)$ and the $\fq$-rat. points stabilized by one of the $4$ subgroups $C_3$ of $\Sym(4)$. Then $\Lambda$ has $7$ distinct $\fq$-rat. points and is a subplane of order $2$; thus $\Lambda=\Pi=\Pi^\prime$ and $H=K$.

\item {\bf Line 6:} $K=G_{P,\ell}\cong E_q^{1+2}:(C_{q-1})^2$, where $P$ and $\ell$ are $\fq$-rational, and $P\in\ell$.
\begin{itemize}
\item If $H=K\cap G_Q$, $Q\in\PG(2,q)\setminus\{P\}$, then either $Q\in\ell$ and $H=G_{P,Q}\cong E_{q^2}:(C_{q-1})^2$; or $Q\notin\ell$ and $H\cong E_q:(C_{q-1})^2$.
\item If $H=K\cap G_r$, $r\ne\ell$ an $\fq$-rat. line, then either $P\in r$ and $H=G_{\ell,r}\cong E_{q^2}:(C_{q-1})^2$; or $P\notin r$ and $H\cong E_q:(C_{q-1})^2$.
\item If $H=K\cap G_T$, $T$ an $\fq$-rat. triangle, then one of the following holds:
\begin{itemize}
\item $P$ is a vertex and $\ell$ is a side of $T$. Then $H\cong (C_{q-1})^2$.
\item $P$ is a vertex of $T$ and $\ell$ is not a side of $T$; or $P$ is not a vertex of $T$ and $\ell$ is a side of $T$. Then $H\cong C_{2(q-1)}$.
\item $P$ is not a vertex but is on a side of $T$, and $\ell$ is not a side of $T$. Then either $H\cong C_{2(q-1)}$, or $H\cong C_2$.
\item $P$ is not on a side of $T$ and $\ell$ is not a side of $T$. Then $H\leq \Sym(3)$.
\end{itemize}
\item If $H=K\cap G_{\tilde T}$, $\tilde T$ an $\fqt\setminus\fq$-rat. triangle, then $H=\{1\}$.
\item If $H=K\cap G_{\Pi}$, $\Pi$ a subplane of order $2$, and $H=\{1\}$, then $P$ and $\ell$ are a point and a line of $\Pi$, and $H\cong D_8$.
\end{itemize}

\item {\bf Line 7:} $K=G_{P,\ell}\cong \GL(2,q)$, where $P$ and $\ell$ are $\fq$-rational, and $P\notin\ell$.
\begin{itemize}
\item If $H=K\cap G_Q$, $Q\in\PG(2,q)\setminus\{P\}$, then either $Q\in\ell$ and $H\cong E_q:(C_{q-1})^2$; or $Q\notin\ell$ and $H\cong E_q: C_{q-1}$ is as in Line 15.
\item If $H=K\cap G_r$, $r\ne\ell$ an $\fq$-rat. line, then either $P\in r$ and $H\cong E_q:(C_{q-1})^2$; or $P\notin r$ and $H\cong E_q: C_{q-1}$ is as in Line 16.
\item If $H=K\cap G_T$, $T$ an $\fq$-rat. triangle, then one of the following holds:
\begin{itemize}
\item $P$ is a vertex of $T$ and $G_{P,T}\cong(C_{q-1})^2: C_2$. Then either $H\cong(C_{q-1})^2: C_2$ or $H\leq C_2$.
\item $P$ is not a vertex but on a side of $T$ and $G_{P,T}\cong C_{2(q-1)}$. Then either $H=\{1\}$, or $H\cong C_{q-1}$ is made by homologies.
\item $P$ is not on a side of $T$ and $G_{P,T}\leq \Sym(3)$. Then $H\leq \Sym(3)$.
\end{itemize}
\item If $H=K\cap G_{\tilde T}$, $\tilde T$ an $\fqt\setminus\fq$-rational triangle, then $H\leq C_3$.
\item If $H=K\cap G_{\Pi}$, $\Pi$ a subplane of order $2$, and $H\ne\{1\}$, then $P$ is a point and $\ell$ a line of $\Pi$, and $H\cong \Sym(3)$.
\end{itemize}

\item {\bf Line 8:} $K=G_{P,Q}\cong E_{q^2}:(C_{q-1})^2$ for some $P,Q\in\PG(2,q)$ with $P\ne Q$.
\begin{itemize}
\item If $H=K\cap G_R$, $R\in\PG(2,q)\setminus\{P,Q\}$, then either $H\cong E_{q^2}: C_{q-1}$ is as in Line 10, or $H\cong(C_{q-1})^2$.
\item If $H=K\cap G_{\ell}$ with $\ell\ne PQ$, and $R=\ell\cap PQ$, then either $R\in\{P,Q\}$ and $H\cong E_q:(C_{q-1})^2$; or $R\notin\{P,Q\}$ and $H\cong E_q: C_{q-1}$ is as in Line 15.
\item If $H=K\cap G_T$, $T$ an $\fq$-rat. triangle, then either $P,Q$ are vertices of $T$ and $H\cong(C_{q-1})^2$; or $P,Q$ are not both vertices but still on the same side of $T$, and $H$ is a group $C_{q-1}$ of homologies; or $H\leq \Sym(3)$.
\item If $H=K\cap G_{\tilde T}$, $\tilde T$ an $\fqt\setminus\fq$-rat. triangle, then $H=\{1\}$.
\item If $H=K\cap G_{\Pi}$, $\Pi$ a subplane of order $2$, then either $\{P,Q\}\not\subset\Pi$ and $H=\{1\}$, or $\{P,Q\}\subset\Pi$ and $H\cong E_4$ is as in line 28.
\end{itemize}

\item {\bf Line 9:} $K=G_{\ell,r}\cong E_{q^2}:(C_{q-1})^2$ for some $\fq$-rat. lines $\ell,r$ with $\ell\ne r$.
Dual arguments with respect to the ones used in the previous point
prove the claim.

\item {\bf Line 10:} $K=G_{P_1,\ldots,P_{q+1}}\cong E_{q^2}: C_{q-1}$ where $P_1,\ldots,P_{q+1}$ are the distinct $\fq$-rat. points of a line $\ell$.
\begin{itemize}
\item If $H=K\cap G_P$, $P\in\PG(2,q)\setminus\ell$, then $H$ is a group $C_{q-1}$ of homologies.
\item If $H=K\cap G_r$, $r\ne\ell$ an $\fq$-rat. line, then $H\cong E_q:(C_{q-1})^2$.
\item If $H=K\cap G_T$, $T$ an $\fq$-rat. triangle, then either $H$ is a group $C_{q-1}$ of homologies, or $H\leq C_2$.
\item If $H=K\cap G_{\tilde T}$, $\tilde T$ an $\fqt\setminus\fq$-rat. triangle, then $H=\{1\}$.
\item If $H=K\cap G_{\Pi}$, $\Pi$ a subplane of order $2$, then either $\ell$ is not a line of $\Pi$ and $H=\{1\}$; or $\ell$ is a line of $\Pi$ and $H\cong E_4$ is as in Line 28.
\end{itemize}

\item
{\bf Line 11:} $K=G_{\ell_1,\ldots,\ell_{q+1}}\cong E_{q^2}:
C_{q-1}$ where $\ell_1,\ldots,\ell_{q+1}$ are concurrent
$\fq$-rat. lines. Dual arguments with respect to Line 10 prove the claim.

\item
{\bf Line 12:} $K=G_{P,Q,\ell,r}\cong E_q : (C_{q-1})^2$, where
$P,Q\in\PG(2,q)$, $P\ne Q$, $r=PQ$, and $\ell\ne r$ is another
$\fq$-rational line with $P\in\ell$.
\begin{itemize}
\item If $H=K\cap G_R$, $R\in\PG(2,q)\setminus\{P,Q\}$, then either $R\in r$ and $H\cong E_q: C_{q-1}$; or $R\in \ell$ and $H\cong (C_{q-1})^2$; or $H$ is a group $C_{q-1}$ of homologies.
\item If $H=K\cap G_s$, $s\notin\{\ell,r\}$ an $\fq$-rat. line, then dual arguments yield either $H\cong E_q: C_{q-1}$, or $H\cong(C_{q-1})^2$, or $H\cong C_{q-1}$ made by homologies.
\item If $H=K\cap G_T$, $T$ an $\fq$-rat. triangle, and $H\not\leq \Sym(3)$, then some $\sigma\in H\setminus\{1\}$ stabilizes $T$ pointwise.
Since $\sigma$ cannot stabilize a projective frame pointwise, the
points $P,Q$ are on some side of $T$ and the lines $\ell,r$ pass
through some vertex of $T$. If $P,Q$ are vertices of $T$ and
$r,\ell$ are sides of $T$, then $H\cong (C_{q-1})^2$; otherwise, $H$
is a group $C_{q-1}$ of homologies.
\item If $H=K\cap G_{\tilde T}$, $\tilde T$ an $\fqt\setminus\fq$-rational triangle, then $H=\{1\}$.
\item If $H=K\cap G_{\Pi}$, $\Pi$ a subplane of order $2$, then $H\leq C_2$.
\end{itemize}

\item
{\bf Line 13:} $K=G_{P,T}\cong(C_{q-1})^2:C_2$, $P$ a vertex of an $\fq$-rational triangle $T$.
\begin{itemize}
\item If $H=K\cap G_Q$, $Q\in\PG(2,q)\setminus\{P\}$, then either $Q$ is a vertex of $T$ and $H\cong(C_{q-1})^2$; or $Q$ is on the side not through $P$ and $H\cong C_{2(q-1)}$; or $Q$ is on a side thorugh $P$ and $H$ is a group $C_{q-1}$ of homologies; or $H\cong C_2$.
\item If $H=K\cap G_\ell$, $\ell$ an $\fq$-rat. line, then dual arguments prove the claim.
\item If $H=K\cap G_{T^{\prime}}$, $T^{\prime}\ne T$ an $\fq$-rat. triangle, and $H\not\leq \Sym(3)$, then either $H\cong C_{2(q-1)}$ or $H$ is a group $C_{q-1}$ of homologies, according to $|T\cap T^{\prime}|=1$ or $|T\cap T^{\prime}|=2$, respectively.
\item If $H=K\cap G_{\tilde T}$, $\tilde{T}$ an $\fqt\setminus\fq$-rat. triangle, then $H=\{1\}$.
\item If $H=K\cap G_{\Pi}\ne\{1\}$, $\Pi$ a subplane of order $2$, then either $H\cong C_2$, or $p=3$ and $H\cong C_7$.
\end{itemize}

\item {\bf Line 14:} $K=G_T\cong(C_{q-1})^2$ for some $\fq$-rational triangle $T$. Let $H=K\cap M$ for some maximal subgroup $M$ of $G$ such that $\{1\}<H<K$.
Then either $p>3$ and $H$ is a group $C_{q-1}$ of homologies; or
$p=3$ and $H\cong C_7$ is as in Line 24.

\item {\bf Line 15:} $K=G_{P_1,\ldots,P_{q+1},\ell}\cong E_q:C_{q-1}$, where $P_1,\ldots,P_{q+1}$ are the $\fq$-rat. points of an $\fq$-rat. line $r$, and $\ell\ne r$ is an $\fq$-rat. line meeting $r$ in $P$.
\begin{itemize}
\item If $H=K\cap G_Q$, $Q\in\PG(2,q)\setminus r$, then either $Q\in\ell$ and $H$ is a group $C_{q-1}$ of homologies, or $Q\notin\ell$ and $H=\{1\}$.
\item If $H=K\cap G_s$, $s\notin\{r,\ell\}$ an $\fq$-rat. line, then either $P\in s$ and $H$ is a group $E_q$ of elations, or $P\notin s$ and $H$ is a group $C_{q-1}$ of homologies.
\item If $H=K\cap G_T$, $T$ an $\fq$-rat. triangle, then either $H$ is a group $C_{q-1}$ of homologies, or $H\leq C_2$.
\item If $H=K\cap G_{\tilde T}$, $\tilde T$ an $\fqt\setminus\fq$-rat. triangle, then $H=\{1\}$.
\item If $H=K\cap G_{\Pi}$, $\Pi$ a subplane of order $2$, then $H=\leq C_2$.
\end{itemize}

\item {\bf Line 16:} $K=G_{\ell_1,\ldots,\ell_{q+1},P}\cong E_q:C_{q-1}$.
Dual arguments with respect to Line 15 prove the claim.


\item {\bf Line 17:} $K\cong C_{2(q-1)}$. Let $H=K\cap M$, for a maximal subgroup $M$ of $G$ with $\{1\}<H<K$. Then either $H\cong C_2$  or $H$ is a group $C_{q-1}$ of homologies.

\item {\bf Line 18:} $K\cong E_q$ is the group of elations with a given axis and center. Let $H=K\cap M$, for a maximal subgroup $M$ such that $\{1\}<H<K$. Then $H\cong C_2$.

\item {\bf Line 19:} $K\cong C_{q-1}$ is the group of homologies with a given axis and center. Let $H=K\cap M$, for a maximal subgroup $M$ such that $H<K$. Then $H=\{1\}$.

\item {\bf Lines 20 to 31:} if $K$ is in one of the Lines 20 to 31, then every subgroup of $K$ is in Table \ref{tabella:teorema}, apart from the subgroup $\Alt(4)$, which has already been shown not to be intersection of maximal subgroups of $G$ (see Line 5).
\end{itemize}
\end{proof}

\begin{proposition}\label{prop:mobius}
For any $H<G$ in Table \ref{tabella:teorema}, $\mu(H)$ is given in Table \ref{tabella:teorema}.
\end{proposition}

\begin{proof}
We make implicit use of what has been shown in the previous
propositions. In particular, the proof of Proposition
\ref{prop:isintersection} contains all the maximal subgroups of $G$
containing $H$. We denote by $n(H)$ the number of such maximal subgroups.
\begin{itemize}
\item {\bf Lines 1 to 5.} Since $H$ is a maximal subgroup of $G$, $\mu(H)=-1$.
\item {\bf Line 6, $H\cong E_q^{1+2}:(C_{q-1})^2$.} Since $n(H)=2$, $\mu(H)=1$.
\item {\bf Line 7, $H\cong\GL(2,q)$.} Arguing as in Line 6, $\mu(H)=1$.
\item {\bf Line 8, $H=G_{P,Q}\cong E_{q^2}:(C_{q-1})^2$.}
We have $n(H)=3$, $H=G_P\cap G_Q$, $H\ne G_P\cap G_{PQ}$, $H\ne
G_Q\cap G_{PQ}$. Thus, $\mu(H)=0$.
\item {\bf Line 9, $H=G_{\ell,r}\cong E_{q^2}:(C_{q-1})^2$.} Arguing as in Line 8, $\mu(H)=0$.
\item {\bf Line 10, $H=G_{P_1,\ldots,P_{q+1}}\cong E_{q^2}:C_{q-1}$.}
Let $\ell$ be the line through $P_1,\ldots,P_{q+1}$. Then $n(H)=q+2$; the overgroups of $H$
in Table \ref{tabella:teorema} are the following:
$G_{P_1},\ldots,G_{P_{q+1}},G_{\ell}$, $q+1$ groups $G_{P_i,\ell}$,
and groups $G_{P_i,P_j}$. Thus, $\mu(H)=0$.

\item {\bf Line 11, $H=G_{\ell_1,\ldots,\ell_{q+1}}\cong E_{q^2}:C_{q-1}$.} Arguing as in Line 10, $\mu(H)=0$.

\item {\bf Line 12, $H=G_{P,Q,r,\ell}\cong E_q:(C_{q-1})^2$.}
Here, $r=PQ$, $P\in\ell$, $Q\notin\ell$, and $P,Q,r,\ell$ are
$\fq$-rational. We have $n(H)=4$, and the overgroups of $H$ in Table
\ref{tabella:teorema} are $G_P$, $G_Q$, $G_{r}$, $G_{\ell}$,
$G_{Q,r}$, $G_{P,r}$, $G_{P,\ell}$, $G_{Q,\ell}$, $G_{P,Q}$,
$G_{r,\ell}$. Thus, $\mu(H)=-1$.

\item {\bf Line 13, $H\cong(C_{q-1})^2:C_2$.} Then $n(H)=3$ and $H=G_P\cap G_{QR}\cap G_T$ for some $\fq$-rational triangle $T=\{P,Q,R\}$.
The overgroups of $H$ in Table \ref{tabella:teorema} are $G_P$,
$G_{QR}$, $G_T$, $G_{P,QR}$. Thus, $\mu(H)=1$.

\item {\bf Line 14, $H\cong(C_{q-1})^2$.}
Then the overgroups of $H$ in Table
\ref{tabella:teorema} are the following: $7$ maximal subgroups of
$G$; $6$ groups of type $E_q^{1+2}:(C_{q-1})^2$; $3$ groups
$\GL(2,q)$; $3$ groups $(C_{q-1})^2:C_2$; $6$ groups
$E_q:(C_{q-1})^2$. Thus, $\mu(H)=0$.


\item {\bf Line 15, $H=G_{P_1,\ldots,P_{q+1},\ell}\cong E_q:C_{q-1}$.}
Here, $P_1,\ldots,P_{q+1}$ are collinear in $r$, and $\ell\ne r$;
let $Q=r\cap\ell$. Then the overgroups of $H$ in
Table \ref{tabella:teorema} are the following: $q+3$ maximal
subgroups of $G$; $2q+2$ groups $E_q^{1+2}:(C_{q-1})^2$ or
$\GL(2,q)$; groups $G_{P_i,P_j}$, $G_{r,\ell}$,
$G_{P_1,\ldots,P_{q+1}}$; $q$ groups $E_q:(C_{q-1})^2$. Thus,
$\mu(H)=0$.

\item {\bf Line 16, $H=G_{\ell_1,\ldots,\ell_{q+1},P}\cong E_q:C_{q-1}$.}
Arguing as in Line 15, $\mu(H)=0$.

\item {\bf Line 17, $H\cong C_{2(q-1)}$.}
Then the overgroups $K$ of $H$ in Table
\ref{tabella:teorema} with $\mu(K)\ne0$ are the following: $q+4$
maximal subgroups of $G$; $4$ groups $E_q^{1+2}:(C_{q-1})^2$ or
$\GL(2,q)$; $1$ group $E_q:(C_{q-1})^2$; $q$ groups $(C_{q-1})^2:
C_2$. Thus, $\mu(H)=0$.

\item {\bf Line 18, $H\cong E_q$.}
Here, $H$ is the group of elations with a given center and axis.
Then the overgroups $K$ of $H$ in Table \ref{tabella:teorema} with
$\mu(K)\ne0$ are the following: $2q+2$ maximal subgroups of $G$;
$(q+1)^2$ groups $E_q^{1+2}:(C_{q-1})^2$ or $\GL(2,q)$; $q^2$ groups
$E_q:(C_{q-1})^2$. Thus, $\mu(H)=0$.

\item {\bf Line 19, $H\cong C_{q-1}$.}
Here, $H$ is the group of homologies with a given center and axis.
Then the overgroups $K$ of $H$ in Table \ref{tabella:teorema} with
$\mu(K)\ne0$ are the following: $n(H)=2q+4+\binom{q+1}{2}$ maximal
subgroups of $G$; $(q+2)^2$ groups $E_q^{1+2}:(C_{q-1})^2$ or
$\GL(2,q)$; $(q+1)^2+(q+1)q$ groups $E_q:(C_{q-1})^2$;
$\binom{q+1}{2}+(q+1)q$ groups $(C_{q-1})^2: C_2$. Thus, $\mu(H)=0$.

\item {\bf Lines 20 to 22:} we have $n(H)=2$ and hence $\mu(H)=1$.




\item {\bf Line 23, $H\cong D_8$.}
The overgroups $K$ of $H$ in Table \ref{tabella:teorema} with
$\mu(K)\ne0$ are the following: $2+\frac{q}{2}$ maximal subgroups of
$G$; $1$ group $E_q^{1+2}:(C_{q-1})^2$; $\frac{q}{2}$ groups
$G_{P,\Pi}\cong\Sym(4)$; $\frac{q}{2}$ groups
$G_{\ell,\Pi}\cong\Sym(4)$. Thus, $\mu(H)=-\frac{q}{2}$.

\item {\bf Line 24, $H\cong C_7$.}
The overgroups $K$ of $H$ in Table \ref{tabella:teorema} with
$\mu(K)\ne0$ are the following: if $p>3$, $1+\frac{q^2+q+1}{7}$
maximal subgroups, $q^2+q+1$ groups $C_7: C_3$; if $p=3$,
$14$ maximal subgroups, $6$ groups $E_q^{1+2}:(C_{q-1})^2$,
$3$ groups $\GL(2,q)$, $6$ groups $E_q:(C_{q-1})^2$, $3$ groups
$(C_{q-1})^2: C_2$, $7$ groups $C_7: C_3$. Thus, $\mu(H)=0$.

\item {\bf Line 25, $H\cong\Sym(3)$.}
The overgroups $K$ of $H$ in Table \ref{tabella:teorema} with
$\mu(K)\ne0$ are the following: $2q$ maximal subgroups of $G$; $1$
group $\GL(2,q)$; $q+1$ groups $G_{P,\Pi}\cong\Sym(4)$; $q+1$ groups
$G_{\ell,\Pi}\cong\Sym(4)$. Thus, $\mu(H)=0$.


\item {\bf Line 26, $H\cong C_4$.}
The overgroups $K$ of $H$ in Table \ref{tabella:teorema} with
$\mu(K)\ne0$ are the following: $\frac{q^2+8}{4}$ maximal subgroups
of $G$; $1$ group $E_q^{1+2}:(C_{q-1})^2$; $\frac{q^2}{4}$ groups
$G_{P,\Pi}\cong\Sym(4)$; $\frac{q^2}{4}$ groups
$G_{\ell,\Pi}\cong\Sym(4)$; $\frac{q}{2}$ groups $D_8$. Thus,
$\mu(H)=0$.

\item {\bf Line 27, $H\cong E_4$, $H\leq G_{\ell_1,\ell_2,\ell_3}$.}
The overgroups $K$ of $H$ in Table \ref{tabella:teorema} with
$\mu(K)\ne0$ are the following: $\frac{q^2+4q+8}{4}$ maximal
subgroups of $G$; $q+1$ groups $E_q^{1+2}:(C_{q-1})^2$;
$\frac{q^2}{4}$ groups $G_{P,\Pi}\cong\Sym(4)$;
$3\cdot\frac{q^2}{4}$ groups $G_{\ell_i,\Pi}\cong\Sym(4)$;
$\frac{3q}{2}$ groups $D_8$. Thus, $\mu(H)=0$.


\item {\bf Line 28, $H\cong E_4$, $H\leq G_{P_1,P_2,P_3}$.}
Arguing as in Line 27, $\mu(H)=0$.


\item {\bf Line 29, $H\cong C_3$.}
The overgroups $K$ of $H$ in Table \ref{tabella:teorema} with
$\mu(K)\ne0$ are the following: $\frac{4q^2+2}{3}$ maximal subgroups; $1$ group $\GL(2,q)$; $\frac{q^2-1}{3}$ groups
$G_{P,\Pi}\cong\Sym(4)$; $\frac{q^2-1}{3}$ groups
$G_{\ell,\Pi}\cong\Sym(4)$; $\frac{2(q^2-1)}{3}$ groups $C_7: C_3$.
Thus, $\mu(H)=0$.

\item {\bf Line 30, $H\cong C_2$.}
The overgroups $K$ of $H$ Table \ref{tabella:teorema} with
$\mu(K)\ne0$ are the following:
$2q+2+\frac{q^3(q-1)}{2}+\frac{q^3(q-1)}{8}$ maximal subgroups of
$G$; $2q+1$ groups $E_q^{1+2}:(C_{q-1})^2$; $q^2$ groups $\GL(2,q)$;
$q^2$ groups $E_q:(C_{q-1})^2$; $\frac{q^3(q-1)}{2}$ groups
$(C_{q-1})^2: C_2$; $3\frac{q^3(q-1)}{8}$ groups
$G_{P,\Pi}\cong\Sym(4)$; $3\frac{q^3(q-1)}{8}$ groups
$G_{\ell,\Pi}\cong\Sym(4)$; $\frac{5q^2(q-1)}{4}$ groups $D_8$.
Thus, $\mu(H)=0$.

\item {\bf Line 31, $H=\{1\}$.}
The overgroups $K$ of $H$ in Table \ref{tabella:teorema} with
$\mu(K)\ne0$ are the following: $q^2+q+1$ groups $G_P\cong
E_{q^2}:\GL(2,q)$; $q^2+q+1$ groups $G_\ell\cong E_{q^2}:\GL(2,q)$;
$\frac{q^3(q+1)(q^2+q+1)}{6}$ groups $(C_{q-1})^2:\Sym(3)$;
$\frac{q^3(q-1)^2(q+1)}{3}$ groups $C_{q^2+q+1}:C_3$;
$\frac{q^3(q^3-1)(q^2-1)}{168}$ groups $\PSL(3,2)$; $(q^2+q+1)(q+1)$
groups $E_q^{1+2}:(C_{q-1})^2$; $(q^2+q+1)q^2$ groups $\GL(2,q)$;
$(q^2+q+1)(q^2+q)q$ groups $E_q:(C_{q-1})^2$;
$\frac{(q^2+q+1)q^3(q+1)}{2}$ groups $(C_{q-1})^2: C_2$;
$\frac{q^3(q^3-1)(q^2-1)}{24}$ groups $G_{P,\Pi}\cong \Sym(4)$;
$\frac{q^3(q^3-1)(q^2-1)}{24}$ groups $G_{\ell,\Pi}\cong \Sym(4)$;
$\frac{q^3(q^3-1)(q^2-1)}{21}$ groups $C_7: C_3$;
$\frac{q^2(q^3-1)(q^2-1)}{4}$ groups $D_8$. Thus, $\mu(H)=0$.
\end{itemize}
\end{proof}

\section{The M\"obius function and other combinatorial objects}\label{sec:connections}

There are several situations in which the knowledge of the M\"obius
function of a group may be of help. We mention an application of in
topological graph theory; see \cite{LLK} for the details. Given a
finite group $G$, define $ d_k:=\frac{1}{|\aut(G)|}\sum_{H\leq
G}\mu(H)|H|^k$. Let $\Gamma$ be a connected simple graph with vertex
set $V(\Gamma)$ and edge set $E(\Gamma)$, and let
$\beta=|E(\Gamma)|-|V(\Gamma)|+1$ be the first Betti number of
$\Gamma$. Let ${\rm Isoc}(\Gamma;G)$ be the number of isomorphism
classes of connected simple graphs $\tilde{\Gamma}$ such that
$\tilde{\Gamma}$ admits $G$ as an automorphism group acting
semiregularly on $\tilde{\Gamma}$ and the quotient graph
$\tilde{\Gamma}/G$ is isomorphic to $\Gamma$. Then ${\rm
Isoc}(\Gamma;G) = d_{\beta}(G)$.


We now give another topological application of the M\"obius function.
Let $(\mathcal{P},\preceq)$ be a finite poset, and $\hat{\cP}$ be
the poset obtained from $\cP$ by adjoining a least element $\hat{0}$
and a greatest element $\hat{1}$. Let $\Delta(\mathcal{P})$ be the
order complex of $\mathcal{P}$, i.e. the simplicial complex whose
vertices are the element of $\cP$ and whose $k$-dimensional faces
are the s $a_0\prec a_1\prec\cdots\prec a_k$ of distinct
elements $a_0,\ldots,a_k\in\cP$. Let $\chi(\Delta(\mathcal{P}))$ be
the Euler characteristic and $\tilde{\chi}(\Delta(\mathcal{P}))$ be
the reduced Euler characteristic of $\Delta(\mathcal{P})$. The
M\"obius function of $\hat{\cP}$ is related to
$\tilde{\chi}(\Delta(\cP))$ as stated in Proposition \ref{prop:reduced}, which essentially restates a result by Hall \cite{Hall} on the computation of $\mu_{\hat{P}}(\hat{0},\hat{1})$ by means of the chains of even and odd length between $\hat{0}$ and $\hat{1}$.




\begin{proposition}{\rm (see \cite[Proposition 3.8.6]{Stanley})}\label{prop:reduced}
Let $\cP$ be a finite poset. Then
$$\tilde{\chi}(\Delta(\cP))= \mu_{\hat{\cP}}(\hat{0},\hat{1})=-\sum_{x\in\cP}\mu_{\hat{\cP}}(\hat{0},x). $$
\end{proposition}

Let $r$ be a prime number and $\cP=L_r$ be the poset of nontrivial $r$-subgroups of a finite group $G$ ordered by inclusion.

\begin{lemma}\label{lemma:euler}\textrm{(\cite[Prop. 2.1]{Quillen})}
Let $H\in L_r$. If $H$ is not elementary abelian, then
$\mu_{\hat{L}_r}(\hat{0},H)=0$. If $H$ is elementary abelian of
order $r^s$, then $\mu_{\hat{L}_r}(\hat{0},H)= (-1)^s
r^{\binom{s}{2}}$.
\end{lemma}

For the rest of this section, $q$ is any prime power and $G$ is the
group $\PGL(3,q)$. Using Proposition \ref{prop:reduced} and Lemma
\ref{lemma:euler}, we determine $\tilde{\chi}(\Delta(L_r))$ for any
prime $r$.

\begin{proposition}\label{prop:euler}
For any prime number $r$, exactly one of the following cases holds:
\begin{itemize}
\item $r\nmid|G|$ and $\tilde{\chi}(\Delta(L_r))=0$;
\item $r\mid q$ and $\tilde{\chi}(\Delta(L_r))=-(q^3-1)$;
\item $r\mid(q^2+q+1)$, $r\ne3$, and $\tilde{\chi}(\Delta(L_r))=\frac{q^3(q-1)^2(q+1)}{3}$;
\item $r\mid(q+1)$, $r\ne2$, and $\tilde{\chi}(\Delta(L_r))=\frac{q^3(q^3-1)}{2}$;
\item $r\mid(q-1)$, $r\notin\{2,3\}$, and $\tilde{\chi}(\Delta(L_r)) = -\frac{q^2(q^2+q+1)(q^2+q-3)}{3}$:
\item $r=2$, $q$ is odd, and $\tilde{\chi}(\Delta(L_2))=-\frac{q^2(q^2+q+1)(q^2+q-3)}{3}$;
\item $r=3$, $3\mid(q-1)$, and $\tilde{\chi}(\Delta(L_3))=-\frac{q^2(q^6-q^4+7q^3-7q-8)}{8}$.
\end{itemize}
\end{proposition}

\begin{proof}
From $\gcd(q-1,q^2+q+1)\in\{1,3\}$ and the order of $G$ follows that the divisibility conditions in the claim are exhaustive and pairwise incompatible.
\begin{itemize}
\item Suppose $r\nmid|G|$. Then $\tilde{\chi}(\Delta(L_r))=\chi(\emptyset)=0$.
\item Suppose $r\mid q$, i.e. $r$ is the characteristic of $\mathbb{F}_q$. Then $H\leq G$ is an elementary abelian $r$-subgroup if and only if $H$ is made by elations with the same center or axis.
Let $N_{ac}(i)$ be the number of subgroups $E_{r^i}$ whose elements
have both the same axis and center. Let $N_a(i)$ (resp. $N_c(i)$) be
the number of subgroups $E_{r^i}$ whose elements have a common axis
(resp. center) but not a common center (resp. axis). By duality,
$N_a(i)=N_c(i)$ for any $i$. The subgroup of all elations with
both the same axis and center is a $E_{q}$; the subgroup
of all elations with a given axis (resp. center) is a
$E_{q^2}$. Using the Gaussian coefficient, this implies
\begin{footnotesize}
$$ \tilde{\chi}(\Delta(L_r))=-\left( \sum_{i=1}^d N_{ac}(i)\cdot(-1)^i r^{\binom{i}{2}} + 2\sum_{i=1}^{2d}N_{a}(i)\cdot(-1)^i r^{\binom{i}{2}} \right) $$
$$ = -\left( (q^2+q+1)(q+1)\sum_{i=1}^d (-1)^i r^{\binom{i}{2}}\binom{d}{i}_r +\right.$$
$$ \left.+2(q^2+q+1)\left(\sum_{i=1}^{2d}(-1)^i
r^{\binom{i}{2}}\binom{d}{i}_r - (q+1)\sum_{i=1}^d (-1)^i
r^{\binom{i}{2}}\binom{d}{i}_r\right) \right). $$
\end{footnotesize}
Using the property $\binom{n}{k}_r
=r^k\binom{n-1}{k}_r+\binom{n-1}{k-1}_r$, the claim follows.

\item Suppose $r\mid(q^2+q+1)$ and $r\ne3$. Then $S_r$ is contained in a maximal subgroup $C_{q^2+q+1}: C_3$ of $G$ and determines it uniquely. Thus $G$ has exactly $[G:(C_{q^2+q+1}: C_3)]$ subgroups $C_r$, 
and
$\tilde{\chi}(\Delta(L_r))=-\frac{q^3(q-1)^2(q+1)}{3}\cdot(-1)$.
\item Suppose $r\mid(q+1)$ and $r\ne2$. Then $S_r$ is contained in a $C_{q+1}$ and hence is cyclic. $C_{q+1}\leq G$ is uniquely determined by its fixed points $P\in\PG(2,q)$ and $Q,R\in\PG(2,q^2)\setminus\PG(2,q)$, where $R=Q^q$. Hence, we have $q^2+q+1$ choices for $P$, and then $q^2$ choices for the line $QR$ and $\frac{q^2-q}{2}$ choices for $\{Q,R\}$ on $QR$.
Thus,
$\tilde{\chi}(\Delta(L_r))=-\frac{(q^2+q+1)q^2(q^2-q)}{2}\cdot(-1)$.
\item Suppose $r\mid(q-1)$ and $r\notin\{2,3\}$. Then $S_r$ is contained in a maximal subgroup $G_T\cong (C_{q-1})^2: \Sym(3)$; hence, the elementary abelian $r$-subgroups of $G$ have size $r$ or $r^2$.
A subgroup $E_{r^2}$ is contained in exactly one group $G_T$; thus,
$G$ has exactly $[G:G_T]$ subgroups
$E_{r^2}$. A subgroup $C_r$ made by homologies is uniquely
determined by its center and axis; hence, $G$ has exactly
$(q^2+q+1)q^2$ subgroups $C_r$ of homologies. A subgroup $C_r$
not made by homologies stabilizes pointwise an $\fq$-rat.
triangle $T$; $G_T$ has exactly $3$ subgroups $C_r$ of
homologies and $\frac{r^2-1}{r-1}-3=r-2$ subgroups $C_r$ not of
homologies. Then $G$ has exactly
$[G: G_T]\cdot(r-2)$ subgroups
$C_r$ not of homologies. Altogether,
\begin{footnotesize}
$$
\tilde{\chi}(\Delta(L_r))=-\left(\frac{q^3(q+1)(q^2+q+1)}{6}\cdot r
- \left( (q^2+q+1)q^2 + \frac{q^3(q+1)(q^2+q+1)(r-2)}{6} \right)
\right). $$
\end{footnotesize}
\item Suppose $r=2$ and $q$ odd. Every involution of $G$ is a homology (see \cite{Mitchell}); hence a subgroup $C_2\leq G$ is uniquely determined by the choice of the center and axis, and there are $(q^2+q+1) q^2$ such choices.
A subgroup $E_4\leq G$ is uniquely determined by the
$\mathbb{F}_q$-rat. triangle fixed pointwise by $E_4$; hence $G$ has exactly
$[G: G_T]$ subgroups $E_4$. No
more than $3$ homologies of the same order in ${\rm PSL}(3,q)$ can
commute pairwise (as they stabilize the center and axis of each
other, and the homology of a given order, center and axis is unique).
Hence, $G$ has no subgroups $E_8$. Thus,
$\tilde{\chi}(\Delta(L_2))=-\left(q^2(q^2+q+1)\cdot(-1)+\frac{q^3(q+1)(q^2+q+1)}{6}\cdot2\right)$.
\item Suppose $r=3\mid(q-1)$. Then $r\mid(q^2+q+1)$, and the elements of order $3$ are either homologies, or stabilize pointwise a triangle.
$G$ has exactly $q^2(q^2+q+1)$ homologies of order $3$, $[G:G_T]$ elements of order $3$ stabilizing an $\mathbb{F}_q$-rat.
triangle $T$,
and $[G:G_{\tilde T}]$ elements of order $3$ stabilizing an
$\mathbb{F}_{q^3}\setminus\mathbb{F}_q$-rat. triangle
$\tilde{T}$.
Altogether, $G$ has
$q^2(q^2+q+1)+\frac{q^3(q+1)(q^2+q+1)}{6}+\frac{q^3(q-1)^2(q+1)}{3}$ elements of order $3$.
Let $H\leq G$ with $H\cong E_9$. Since $9\nmid(q^2+q+1)$, $H$ is
contained in a max. subgroup $G_T\cong(C_{q-1})^2:\Sym(3)$. This implies in particular that $G$ has no subgroups $E_{27}$.
Suppose that $H$ contains a homology, with center $C$. Then $C$ is a
vertex of $T$, and $H$ is the unique $E_9\leq G$
stabilizing $T$ pointwise. Thus, the number of $E_9\leq G$
containing a homology is $[G:G_T]$.
Suppose that $H$ does not contain any homology. It is easily seen
that there exists $\sigma\in H\setminus\{1\}$ which stabilizes $T$ pointwise.
Hence, $H=\langle\sigma,\tau\rangle$ where $\sigma$ and $\tau$ act
respectively trivially and with a $3$-cycle on $T$. Up to
conjugation, $T$ is the fundamental triangle and
$\sigma=\diag(\lambda,\lambda^2,1)$, where $\lambda\in\mathbb{F}_q$
has order $3$; up to replacing $\tau$ with $\tau^2$, $\tau:(X:Y:Z)\mapsto(\mu Y : \rho Z : X)$ with
$\mu,\rho\in\mathbb{F}_q^*$.
The short orbits of $H$ in ${\rm PG}(2,\overline{\mathbb{F}}_q)$ are
exactly the four distinct triangles $T,T_1,T_2,T_3$ stabilized
pointwise by some $C_3\leq H$. By direct checking,
$T_1,T_2,T_3$ are either $\mathbb{F}_{q}$-rat. or
$\mathbb{F}_{q^3}\setminus\mathbb{F}_{q}$-rat., according to
$3\mid\frac{q-1}{o(\mu\rho)}$ or $3\nmid\frac{q-1}{o(\mu\rho)}$,
respectively. Therefore, $\langle\sigma\rangle$ is contained in
exactly $\frac{1}{3}\cdot\frac{(q-1)^2}{3}$ subgroups $E_9\leq G$
which stabilize four $\mathbb{F}_q$-rat. triangles; and
$\langle\sigma\rangle$ is contained in exactly
$\frac{1}{3}\cdot\frac{2(q-1)^2}{3}$ subgroups $E_9\leq G$ which
stabilize one $\mathbb{F}_q$- and three
$\mathbb{F}_{q^3}\setminus\mathbb{F}_q$-rat. triangles.
Viceversa, any subgroup $E_9\leq G$ stabilizing only
$\mathbb{F}_q$-rat. triangle (resp. one $\mathbb{F}_q$- and
three $\mathbb{F}_{q^3}\setminus\mathbb{F}_q$-rat. triangles)
contains exactly $4$ subgroups (resp. $1$ subgroup) $C_3$
stabilizing pointwise an $\mathbb{F}_q$-rat. triangle. Then, a
double counting argument shows that $G$ contains exactly:
$[G:G_T]\cdot\frac{(q-1)^2}{9}\cdot\frac{1}{4}$
subgroups $E_9$ with no homology and only $\mathbb{F}_q$-rat.
fixed triangles;
$[G: G_T]\cdot\frac{2(q-1)^2}{9}$ subgroups $E_9$
stabilizing an $\mathbb{F}_{q^3}\setminus\mathbb{F}_q$-rat.
triangle.
The claim now follows by direct computation.
\end{itemize}
\end{proof}


\section*{Acknowledgments}

The research of F. Dalla Volta and G. Zini was supported by the
Italian National Group for Algebraic and Geometric Structures and
their Applications (GNSAGA - INdAM).

\begin{footnotesize}
\begin{center}
\begin{table}[!ht]
\begin{footnotesize}
\caption{Conj. classes of subgroups $H\leq G=\PSL(3,q)$,
$q=2^{p}$, $p$ an odd prime,  with
$\mu(H)\ne0$}\label{tabella:mobius}
\def\arraystretch{1.3}
\begin{tabular}{|c|c|c|c|c|}
\hline $H$ &$\mathcal{C}_i$ & elements of the plane stabilized by
$H$ & $N_G(H)$ & $\mu(H)$ \\ \hline
$G$ & - &  - & $H$ & $1$ \\
$E_{q^2} : \GL(2,q)$ & $\mathcal{C}_1 $ & an $\mathbb{F}_q$-rat. point & $H$ & $-1$ \\
$E_{q^2} : \GL(2,q)$ &$\mathcal{C}_1 $ &  an $\mathbb{F}_q$-rat. line & $H$ & $-1$ \\
$(C_{q-1})^2: {\rm Sym}(3)$ & $\mathcal{C}_2$ & an $\mathbb{F}_q$-rat. triangle & $H$ & $-1$ \\
$C_{q^2+q+1}: C_3$ & $\mathcal{C}_3$ &  an $\mathbb{F}_{q^3} \setminus\mathbb{F}_{q}$-rat. triangle & $H$ & $-1$ \\
$\PSL(3,2)$ &  $\mathcal{C}_5$ & a subplane of order $2$ & $H$ & $-1$ \\
$E_q^{1+2}:(C_{q-1})^2$ &  $\mathcal{C}_1$ (N) & an $\mathbb{F}_q$-rat. point $P$ and an $\mathbb{F}_q$-rat. line $\ell$, $P\in\ell$ & $H$ & $1$ \\
$\GL(2,q)$ &  $\mathcal{C}_1$ (N) & an $\mathbb{F}_q$-rat. point $P$ and an $\mathbb{F}_q$-rat. line $\ell$, $P\notin\ell$ & $H$ & $1$ \\
$E_q : (C_{q-1})^2$ &  $\mathcal{C}_1$ (N) & two $\mathbb{F}_q$-rat. points and two $\mathbb{F}_q$-rat. lines & $H$ & $-1$ \\
$(C_{q-1})^2: C_2$ &   $\mathcal{C}_1,\mathcal{C}_2$ (N) &  an $\mathbb{F}_q$-rat. triangle and one of its vertexes & $H$ & $1$ \\
$\Sym(4)$ & $\mathcal{C}_1,\mathcal{C}_5$ (N) & a subp. $\Pi$ of order $2$ and a point of $\Pi$ & $H$ & $1$ \\
$\Sym(4)$ &  $\mathcal{C}_1,\mathcal{C}_5$ (N) & a subp. $\Pi$ of order $2$ and a line of $\Pi$ & $H$ & $1$ \\
$C_7: C_3$ &  $\mathcal{C}_2,\mathcal{C}_5$ (N) & a subp. $\Pi$ of order $2$ and a triangle not in $\Pi$ & $H$ & $1$ \\
$D_8$ &   $\mathcal{C}_1,\mathcal{C}_5$ (N)  & a subp. $\Pi$ of order $2$, a point $P$ and a line $\ell$, $P\in\ell$ & $E_q\, . \, E_4$ & $-\frac{q}{2}$ \\
\hline
\end{tabular}
\end{footnotesize}
\end{table}
\end{center}
\end{footnotesize}

\begin{footnotesize}
\begin{center}
\begin{table}[!ht]
\begin{footnotesize} \caption{Subgroups $H$ of $G=\PSL(3,4)$ with
$\mu(H)\ne0$}\label{tabella:mobius4}
\def\arraystretch{1.3}
\begin{tabular}{|c|c|c||c|c|c||c|c|c|}
\hline $H\cong$ & conj. cl. & $\mu_G(H)$ & $H\cong$ & conj. cl. &
$\mu(H)$ &$H\cong$ & conj. cl. & $\mu(H)$ \\ \hline
$G$ & $1$ & $1$ & $E_9: C_4$ & $3$ & $2$ & $D_8$ & $3$ & $-4$ \\
$E_{16}\,.\,\SL(2,4)$ & $2$ & $-1$ & $\Sym(4)$ & $6$ & $2$ & $\Sym(3)$ & $1$ & $-14$ \\
$\Alt(6)$ & $3$ & $-1$ & $C_7: C_3$ & $1$ & $2$ & $C_4$ & $3$ & $-8$ \\
$\PSL(3,2)$ & $3$ & $-1$ & $\Alt(4)$ & $6$ & $-2$ & $C_3$ & $1$ & $24$ \\
$\PSU(3,2)$ & $1$ & $-1$ & $\Alt(4)$ & $1$ & $-1$ & $C_2$ & $1$ & $544$ \\
$E_4^{1+2}: C_3$ & $1$ & $1$ & $D_{10}$ & $1$ & $-3$ & $\{1\}$ & $1$ & $-120960$ \\
$\Alt(5)$ & $7$ & $1$ & $Q_8$ & $1$ & $2$ & & & \\
\hline
\end{tabular}
\end{footnotesize}
\end{table}
\end{center}
\end{footnotesize}

\begin{small}
\begin{center}
\begin{table}[!ht]
\begin{footnotesize}
\caption{Reduced Euler characteristic of the order complex of the
poset of $r$-subgroups of $\PGL(3,q)$}\label{tabella:euler}
\def\arraystretch{1.3}
\begin{tabular}{|c|c||c|c|}
\hline
prime $r$ & $\tilde{\chi}(\Delta(L_r))$ & prime $r$ & $\tilde{\chi}(\Delta(L_r))$ \\
\hline
$r\nmid|PGL(3,q)|$ & $0$ & $r\mid(q-1)$, $r\notin\{2,3\}$ & $-\frac{q^2(q^2+q+1)(q^2+q-3)}{3}$ \\
$r\mid q$ & $-(q^3-1)$ & $r=2\mid(q-1)$ & $-\frac{q^2(q^2+q+1)(q^2+q-3)}{3}$ \\
$3\ne r\mid(q^2+q+1)$ & $\frac{q^3(q-1)^2(q+1)}{3}$ & $r=3\mid(q-1)$ & $-\frac{q^2(q^6-q^4+7q^3-7q-8)}{8}$  \\
$2\ne r\mid(q+1)$ & $\frac{q^3(q^3-1)}{2}$ & & \\
\hline
\end{tabular}
\end{footnotesize}
\end{table}
\end{center}
\end{small}

\begin{small}
\begin{center}
\begin{table}[!ht]
\begin{footnotesize}
\caption{$\mu(H)$ for any $H\leq G$ which is intersection of maximal
subgroups of $G$}\label{tabella:teorema}
\def\arraystretch{1.3}
\begin{tabular}{|c|c|c|c|}
\hline
& $H$ & $N_G(H)$ & $\mu(H)$ \\
\hline
{\bf Line 1}& $E_{q^2}:\GL(2,q)=G_P$ & $H$ & $-1$ \\
{\bf Line 2}&$E_{q^2}:\GL(2,q)=G_{\ell}$ & $H$ & $-1$ \\
{\bf Line 3}&$(C_{q-1})^2:{\rm Sym}(3)=G_T$ & $H$ & $-1$ \\
{\bf Line 4}&$C_{q^2+q+1}:C_3=G_{\tilde{T}}$ & $H$ & $-1$ \\
{\bf Line 5}&$\PSL(3,2)=G_{\Pi}$ & $H$ & $-1$ \\
{\bf Line 6}&$E_q^{1+2}:(C_{q-1})^2$ & $H$ & $1$ \\
{\bf Line 7}&$\GL(2,q)$ & $H$ & $1$\\
{\bf Line 8}&$E_{q^2}:(C_{q-1})^2=G_{P,Q}$ & $H:C_2$ & $0$ \\
{\bf Line 9}&$E_{q^2}:(C_{q-1})^2=G_{\ell,r}$ & $H:C_2$ & $0$ \\
{\bf Line 10}&$E_{q^2}:C_{q-1}=G_{P_1,\ldots,P_{q+1}}$ & $E_{q^2}:\GL(2,q)$ & $0$ \\
{\bf Line 11}&$E_{q^2}:C_{q-1}=G_{\ell_1,\ldots,\ell_{q+1}}$ & $E_{q^2}:\GL(2,q)$ & $0$ \\
{\bf Line 12}&$E_q:(C_{q-1})^2$ & $H$ & $-1$ \\
{\bf Line 13}&$(C_{q-1})^2:C_2$ & $H$ & $1$ \\
{\bf Line 14}&$(C_{q-1})^2$ & $(C_{q-1})^2:{\rm Sym}(3)$ & $0$ \\
{\bf Line 15}&$E_q:C_{q-1}=G_{P_1,\ldots,P_{q+1},\ell}$ & $ E_{q^2}:(C_{q-1})^2 $ & $0$ \\
{\bf Line 16}&$E_q:C_{q-1}=G_{\ell_1,\ldots,\ell_{q+1},P}$ & $E_{q^2}:(C_{q-1})^2$ & $0$ \\
{\bf Line 17}&$C_{2(q-1)}$ & $E_q\times C_{q-1}$ & $0$ \\
{\bf Line 18}&$E_q= G_{P_1,\ldots,P_{q+1},\ell_1,\ldots,\ell_{q+1}}$ & $E_q^{1+2}:(C_{q-1})^2$ & $0$ \\
{\bf Line 19}&$C_{q-1}= G_{P_1,\ldots,P_{q+1},\ell_1,\ldots,\ell_{q+1}}$ & $\GL(2,q)$ & $0$ \\
{\bf Line 20}&${\rm Sym}(4)=G_{P,\Pi}$ & $H$ & $1$ \\
{\bf Line 21}&${\rm Sym}(4)=G_{\ell,\Pi}$ & $H$ & $1$ \\
{\bf Line 22}&$C_7:C_3$ & $H$ & $1$ \\
{\bf Line 23}&$D_8$ & $E_q\,.\,E_4$ & $-\frac{q}{2}$ \\
{\bf Line 24}&$C_7\leq G_{T,\Pi}$ & $\begin{cases} C_{q^2+q+1}:C_3 & \textrm{if } p>3 \\ (C_7)^2:C_3 & \textrm{if } p=3 \end{cases}$ & $0$ \\
{\bf Line 25}&${\rm Sym}(3)$ & $H\times C_{q-1}$ & $0$ \\
{\bf Line 26}&$C_4$ & $E_q\,.\,E_{2q}$ & $0$ \\
{\bf Line 27}&$E_4\leq G_{\ell_1,\ell_2,\ell_3}$ & $E_{q^2}:{\rm Sym}(3)$ & $0$ \\
{\bf Line 28}&$E_4\leq G_{P_1,P_2,P_3}$ & $E_{q^2}:{\rm Sym}(3)$ & $0$ \\
{\bf Line 29}&$C_3$ & $C_{q^2-1}:C_2$ &$0$ \\
{\bf Line 30}&$C_2$ & $E_q^{1+2}:C_{q-1}$ & $0$ \\
{\bf Line 31}&$\{1\}$ & $G$ & $0$ \\
\hline
\end{tabular}
\end{footnotesize}
\end{table}
\end{center}
\end{small}


\end{document}